\documentclass[reqno]{amsart}

\usepackage{a4wide}
\usepackage{color}
\usepackage{mathrsfs}
\usepackage{mathtools}
\usepackage{tikz-cd}
\usepackage{amsmath}
\usepackage{amssymb}
\usepackage{bbm}
\numberwithin{equation}{section}
\usepackage[colorlinks,citecolor=green,linkcolor=red]{hyperref}
\usepackage{hyphenat}
\usepackage[latin1]{inputenc}

\newcommand{\N}{\mathbb{N}}
\newcommand{\R}{\mathbb{R}}
\newcommand{\sfd}{{\sf d}}
\renewcommand{\d}{{\mathrm d}}

\newcommand{\restr}[1]{\lower3pt\hbox{\(|_{#1}\)}}

\newcommand{\nchi}{{\raise.3ex\hbox{\(\chi\)}}}
\newcommand{\Der}{{\rm Der}}
\newcommand{\1}{\mathbbm 1}
\newcommand{\fr}{\penalty-20\null\hfill\(\blacksquare\)}
\newcommand{\mm}{\mathfrak{m}}
\newcommand{\X}{{\rm X}}

\newcommand{\rmC}{{\rm C}}
\newcommand{\LIP}{{\rm LIP}}
\newcommand{\Lip}{{\rm Lip}}
\newcommand{\lip}{{\rm lip}}
\renewcommand{\div}{{\rm div}}

\newtheorem{theorem}{Theorem}[section]
\newtheorem{corollary}[theorem]{Corollary}
\newtheorem{lemma}[theorem]{Lemma}
\newtheorem{proposition}[theorem]{Proposition}
\newtheorem{definition}[theorem]{Definition}
\newtheorem{example}[theorem]{Example}
\newtheorem{remark}[theorem]{Remark}

\title[Functions of bounded variation and Lipschitz algebras in metric measure spaces]{Functions of bounded variation and \\ Lipschitz algebras in metric measure spaces}

\author{Enrico Pasqualetto}
\address{Department of Mathematics and Statistics,
P.O.\ Box 35 (MaD), FI-40014 University of Jyvaskyla}
\email{enrico.e.pasqualetto@jyu.fi}

\author{Giacomo Enrico Sodini}
\address{Institut fur Mathematik - Fakult\"{a}t f\"{u}r Mathematik - Universit\"{a}t Wien,
Oskar-Morgenstern-Platz 1, 1090 Wien (Austria)}
\email{giacomo.sodini@univie.ac.at}

\begin{document}
\date{\today} 
\begin{abstract}
Given a unital algebra $\mathscr A$ of locally Lipschitz functions defined over a metric measure space $({\rm X},{\sf d},\mathfrak m)$,
we study two associated notions of function of bounded variation and their relations: the space ${\rm BV}_{\rm H}({\rm X};\mathscr A)$,
obtained by approximating in energy with elements of $\mathscr A$, and the space ${\rm BV}_{\rm W}({\rm X};\mathscr A)$, defined through
an integration-by-parts formula that involves derivations acting in duality with $\mathscr A$. Our main result provides a sufficient
condition on the algebra $\mathscr A$ under which ${\rm BV}_{\rm H}({\rm X};\mathscr A)$ coincides with the standard metric BV space
${\rm BV}_{\rm H}({\rm X})$, which corresponds to taking as $\mathscr A$ the collection of all locally Lipschitz functions. Our result
applies to several cases of interest, for example to Euclidean spaces and Riemannian manifolds equipped with the algebra
of smooth functions,
or to Banach and Wasserstein spaces equipped with the algebra of cylinder functions.
Analogous results for metric Sobolev spaces ${\rm H}^{1,p}$ of exponent $p\in(1,\infty)$ were previously obtained by several different authors.
\end{abstract}
\maketitle
\tableofcontents
\section{Introduction}
\subsection{General overview}
For more than two decades, functions of bounded variation and sets of finite perimeter have been studied in the general setting of metric measure spaces,
starting from the paper \cite{Mir:03}. Rather refined results are available on those metric measure spaces (often called PI spaces) that are doubling
and support a weak Poincar\'{e} inequality (see e.g.\ \cite{Mir:03,Amb:01,Amb:02,Kin:Kor:Lor:Sha:13,Kin:Kor:Sha:Tuo:14,Lah:20}), and even deeper
structural properties have been proven when also a lower synthetic Ricci curvature bound is imposed, in the class of the so-called \({\sf RCD}(K,N)\)
spaces \cite{ambrosio2018rigidity,bru2019rectifiability,BruPasSem21-constantcodimension,Bre:Gig:23,Ant:Bre:Pas:24}.
However, a fully consistent theory of BV functions does not require any additional assumptions on the metric measure space under
consideration, and that will be our object of study in the present paper.
\medskip

The first notion of metric BV space, introduced by Miranda Jr.\ in \cite{Mir:03}, is formulated in terms of an energy approximation by locally Lipschitz
functions. However, the algebra of locally Lipschitz functions depends solely on the metric structure of the ambient space, and as such it is not capable
of detecting some additional features that the underlying metric space may have. Due to this reason, in several cases of interest -- which we shall briefly
mention in Section \ref{s:motivations} -- it is desirable to know that a smaller, specific algebra of approximating functions can be used to define
the BV space. This is exactly the primary goal of the present paper: to give a sufficient (and effective) condition on some algebra of functions in order
that the corresponding BV space (obtained via approximation) coincides with the original one that was defined in terms of locally Lipschitz functions.
\medskip

More than ten years after the BV space via approximation was developed, Ambrosio and Di Marino introduced in \cite{Amb:DiMa:14} another notion of metric
BV space and showed its equivalence with the former. Their definition is expressed in terms of the behaviour of functions along suitably-selected curves,
where the exceptional curve families are detected using the so-called test plans. A third approach was then proposed by Di Marino in \cite{DiMar:14,DiMaPhD:14},
by means of an integration-by-parts formula involving an appropriate notion of derivation with divergence. Also the latter approach turned out to be
equivalent to the first two. Generalised versions of the BV space via derivations will have a key role in this paper, as we shall discuss more in details later.
Finally, we mention that a fourth notion of metric `Newtonian-type' BV space (that we will not employ in this paper) was introduced more recently by Martio
in \cite{Martio16,Martio16-2}. His notion is based upon the behaviour of functions along \(AM\)-almost every rectifiable curve, where \(AM\) denotes the
so-called approximation modulus. The full equivalence of this latter approach with the other ones we discussed above was proved in \cite{Nob:Pas:Sch:22}.
\subsection{Contents of the paper}
Let \((\X,\sfd,\mm)\) be a metric measure space (as in Definition \ref{def:mms} below). For any locally Lipschitz function \(f\colon\X\to\R\), we denote
by \(\lip_a(f)\colon\X\to[0,+\infty)\) its asymptotic slope function (see \eqref{eq:def_lip_a}), which assigns to each point \(x\in\X\) the `infinitesimal
Lipschitz constant' of \(f\) at \(x\). The algebra of test functions we will consider in the various definitions of BV space is that of bounded
locally Lipschitz functions whose asymptotic slope is integrable, which we denote by \(\LIP_\star(\X)\); see Definition \ref{def:algebra_LIP_star}.
Given any unital separating subalgebra \(\mathscr A\) of \(\LIP_\star(\X)\), we define:
\begin{itemize}
\item The space \({\rm BV}_{\rm H}(\X;\mathscr A)\) of those functions \(f\in L^1(\mm)\) that can be approximated in energy by elements of \(\mathscr A\);
see Definition \ref{def:bvh}. Any given \(f\in{\rm BV}_{\rm H}(\X;\mathscr A)\) is associated with a quantity \(\|{\bf D}f\|_{*,\mathscr A}\in[0,+\infty)\),
which we call its total variation (see \eqref{eq:bvh_tv}). In the distinguished case \(\mathscr A=\LIP_\star(\X)\), where we use the shorter notations
\({\rm BV}_{\rm H}(\X)\) and \(\|{\bf D}f\|_*\), our definition is consistent with the one of \cite{Mir:03,Amb:DiMa:14}. Whereas each function \(f\in{\rm BV}_{\rm H}(\X)\)
is also naturally associated with a total variation measure \(|{\bf D}f|_*\) satisfying \(|{\bf D}f|_*(\X)=\|{\bf D}f\|_*\) (see Definition \ref{def:BV_H_meas}),
we do not know whether it is possible to define some total variation measure associated to an arbitrary function \(f\in{\rm BV}_{\rm H}(\X;\mathscr A)\) when
\(\mathscr A\neq\LIP_\star(\X)\). Indeed, the fact that the set-valued function \(|{\bf D}f|_*\) as in Definition \ref{def:BV_H_meas} is actually
a \(\sigma\)-additive measure is due to the fact that \(\LIP_\star(\X)\) consists of a sufficiently vast class of locally Lipschitz functions, for which the results we will present in Appendix \ref{app:measure} are satisfied.
\item The space \({\rm BV}_{\rm W}(\X;\mathscr A)\) of those functions \(f\in L^1(\mm)\) that satisfy an integration-by-parts formula
in duality with the space \({\rm Der}^\infty_\infty(\X;\mathscr A)\) of Lipschitz \(\mathscr A\)-derivations \(b\colon\mathscr A\to L^1(\mm)\)
having divergence \(\div(b)\in L^\infty(\mm)\); see Definitions \ref{def:der}, \ref{def:div} and \ref{def:bvw}. Any \(f\in{\rm BV}_{\rm W}(\X;\mathscr A)\)
comes with a corresponding total \(\mathscr A\)-variation measure \(|{\bf D}f|_{\mathscr A}\), as it is shown in Proposition \ref{prop:equiv_|Df|}.
In the special case \(\mathscr A=\LIP_\star(\X)\), we just write \({\rm BV}_{\rm W}(\X)\) and \(|{\bf D}f|\) for the sake of brevity. The original definition in \cite{DiMar:14,DiMaPhD:14} corresponds to choosing as \(\mathscr A\) the space of all boundedly-supported
Lipschitz functions.
\end{itemize}
It is worth pointing out that -- differently from most of the previous literature on the topic -- in this paper we are not assuming completeness nor
separability of the ambient metric space \((\X,\sfd)\). The possible lack of separability does not cause major issues, since we still require the reference
measure \(\mm\) to be Radon (recall Definition \ref{def:mms}). On the other hand, allowing for non-complete metric spaces leads to additional
difficulties, as we will see in the sequel. We prefer to consider possibly non-complete spaces in order to include e.g.\ open domains in a given
ambient space. Non-separable spaces are also interesting e.g.\ in view of the potential generalisation of the theories we consider here to the setting of
extended metric-topological measure spaces \cite{AmbrosioErbarSavare16,Sav:22}. Let us also highlight that here we assume \(\mm\) to be a finite measure.
This choice was made mostly for convenience, as it allows us to work with unital subalgebras of \(\LIP_\star(\X)\); this would not be possible for infinite
reference measures, because constant functions are integrable only with respect to a finite measure. Extending our definitions and results to metric spaces
equipped with a (possibly infinite) boundedly-finite Radon measure would be interesting, but outside the scope of this paper.

\medskip
The main results of the present paper can be briefly summarised as follows:
\begin{itemize}
\item For any metric measure space \((\X,\sfd,\mm)\) and any unital separating subalgebra \(\mathscr A\subseteq\LIP_\star(\X)\), we prove in Theorem \ref{thm:BV_W_in_BV_H} that
\[
{\rm BV}_{\rm W}(\X;\mathscr A)\subseteq{\rm BV}_{\rm H}(\X;\mathscr A),
\]
as well as the inequality \(\|{\bf D}f\|_{*,\mathscr A}\leq|{\bf D}f|_{\mathscr A}(\X)\) for every \(f\in{\rm BV}_{\rm W}(\X;\mathscr A)\).
The verification of this statement is based on a rather general duality argument in Convex Analysis, which is inspired by the proof of \cite[Theorem 5.4]{Pas:Tai:25} (and \cite[Theorem 3.3]{Luc:Pas:24}), where metric Sobolev spaces are involved instead.
\item Assuming either that \((\X,\sfd)\) is complete or that its associated topological space is a Radon space, we obtain in Theorem \ref{thm:equivalence_Lip} the identification
\[
{\rm BV}_{\rm H}(\X)={\rm BV}_{\rm W}(\X)
\]
together with the identity of measures \(|{\bf D}f|=|{\bf D}f|_*\) for every \(f\in{\rm BV}_{\rm H}(\X)\). At this level of generality, this result seems to be new.
In the case where \((\X,\sfd)\) is complete, we can also deduce from the results of \cite{Amb:DiMa:14} that every function \(f\in{\rm BV}_{\rm H}(\X)\) can be approximated
in energy by a sequence of bounded \emph{globally} Lipschitz functions; see Proposition \ref{prop:approx_with_glob_Lip}.
Here, the completeness assumption cannot be dropped, as we will showcase in Example \ref{ex:no_glob_approx}.
\item Assuming that \((\X,\sfd)\) is complete and \(\mathscr A\) is a \emph{good algebra}, we prove in Theorem \ref{thm:equivalence_A} that
\begin{equation}\label{eq:intro_equiv_BV_H}
{\rm BV}_{\rm H}(\X;\mathscr A)={\rm BV}_{\rm H}(\X),
\end{equation}
with \(\|{\bf D}f\|_{*,\mathscr A}=\|{\bf D}f\|_*\) for every \(f\in{\rm BV}_{\rm H}(\X)\). In other words, being a good algebra is a sufficient condition for \(\mathscr A\)
to be dense in energy in \({\rm BV}_{\rm H}(\X)\). By a `good algebra' we mean a unital separating subalgebra \(\mathscr A\) of \(\LIP_\star(\X)\) consisting of bounded
globally Lipschitz functions with the property that truncated distance functions from a point can be `approximated well in \({\rm H}^{1,1}\)' by elements of \(\mathscr A\),
where the Sobolev space \({\rm H}^{1,1}(\X;\mathscr A)\) we consider is the one discussed in Definition \ref{def:H11}; for details, we refer to Definition \ref{def:good_algebra},
where we will introduce the concept of good algebra. The proof of \eqref{eq:intro_equiv_BV_H} relies on the key approximation result Proposition \ref{prop:approx_with_A}, where
the assumption of \(\mathscr A\) being a good algebra is used in an essential way. Along the way, we also obtain that \({\rm BV}_{\rm H}(\X;\mathscr A)={\rm BV}_{\rm W}(\X;\mathscr A)\).
\end{itemize}
\subsection{Motivations and comparison with the previous literature}\label{s:motivations}
In studying functions of bounded variation defined using different algebras of locally Lipschitz functions, we have been strongly inspired
by the works \cite{Sav:22,Fo:Sa:So:23}, where unital separating subalgebras \(\mathscr A\) of bounded Lipschitz functions were used to define
metric Sobolev spaces \({\rm H}^{1,p}(\X;\mathscr A)\) of exponent \(p\in(1,\infty)\). It is shown in \cite{Sav:22} that -- in the more general
framework of extended metric-topological measure spaces -- if \(\mathscr A\) is a \emph{compatible algebra} (cf.\ with Example \ref{ex:good_alg}),
then \({\rm H}^{1,p}(\X;\mathscr A)\) coincides with the `standard' Sobolev space \({\rm H}^{1,p}(\X)\) defined via approximation with bounded
Lipschitz functions. Later on, a characterisation of those Lipschitz algebras \(\mathscr A\) for which \({\rm H}^{1,p}(\X;\mathscr A)={\rm H}^{1,p}(\X)\)
has been provided in \cite{Fo:Sa:So:23} (in the case where the ambient space is a metric measure space). Our main Theorem \ref{thm:equivalence_A}
can be regarded as a version of the above result for metric BV spaces, but with some differences:
\begin{itemize}
\item We give only a sufficient condition for the identification \({\rm BV}_{\rm H}(\X;\mathscr A)={\rm BV}_{\rm H}(\X)\), namely the fact that \(\mathscr A\)
is a good algebra. Conversely, one can easily realise (see Proposition \ref{prop:weakly_good_subalgebra}) that a necessary condition for \({\rm BV}_{\rm H}(\X;\mathscr A)={\rm BV}_{\rm H}(\X)\)
to hold is that \(\mathscr A\) is a \emph{weakly good subalgebra}, which we define in a similar way as good algebras, but where the approximation of truncated distance
functions is done in BV (as opposed to \({\rm H}^{1,1}\)). A natural guess is that \(\mathscr A\) being a weakly good subalgebra is in fact
\emph{equivalent} to \({\rm BV}_{\rm H}(\X;\mathscr A)={\rm BV}_{\rm H}(\X)\). However, we were not able to prove it, and thus we leave
it as an open problem. Another guess would be that \(\mathscr A\) being a good algebra is equivalent to \({\rm H}^{1,1}(\X;\mathscr A)={\rm H}^{1,1}(\X)\),
but this problem is -- arguably -- even more difficult, due to the functional-analytic complications one encounters
when studying metric Sobolev spaces of exponent \(p=1\) (cf.\ with \cite{Amb:Iko:Luc:Pas:24}).
\item Our proof strategy for Theorem \ref{thm:equivalence_A} is different from the one of \cite{Sav:22,Fo:Sa:So:23}. In the latter works, the identification
results are obtained directly at the level of the spaces \({\rm H}^{1,p}(\X;\mathscr A)\), making use of (generalised) Hopf--Lax flow techniques. In this paper,
instead, we first prove equivalence results for metric BV spaces defined using derivations, and then we show their equivalence with the BV spaces defined
via approximation. As a theory of derivations on extended metric-topological measure spaces has been recently developed
in \cite{Pas:Tai:25}, we expect that it is possible to extend many of the notions and results of this paper to that setting.
\end{itemize}
We conclude the introduction by discussing some motivations behind our interest in good algebras. As we will observe in Example \ref{ex:good_alg},
the family of good algebras includes all the compatible algebras considered by Savar\'{e} in \cite{Sav:22}, thus in particular the algebra of smooth
functions in the Euclidean space (or in any Riemannian manifold), as well as the algebra of \emph{cylinder functions} in a Banach space or in
Wasserstein/Hellinger/Hellinger--Kantorovich spaces. As it is evident e.g.\ from the works \cite{DiMar:Gig:Pas:Sou:20,Sav:22,Fo:Sa:So:23,Sod:23, Ds:So:25},
whenever the ambient space possesses `nice' structural features, in order to investigate fine analytic properties it is of fundamental importance to detect
a family of regular functions, which are associated with a `geometric' notion of differential. We point out that Theorem \ref{thm:equivalence_A}
was already known to be true in two cases: for \(\X=\R^n\) with the Euclidean norm, \(\mm\) a Radon measure on \(\R^n\) and \(\mathscr A=C^\infty_c(\R^n)\)
the algebra of compactly-supported smooth functions, see \cite[Theorem 5.7]{Gel:Luc:23}; for \(\X=\mathbb B\) a separable Banach space,
\(\mm\) a finite Borel measure on \(\mathbb B\) and \(\mathscr A={\rm Cyl}(\mathbb B)\) the algebra of smooth cylindrical functions, see \cite[Theorem 4.1]{Pas:24-3}.
\subsection*{Acknowledgements}
The first named author was supported by the Research Council of Finland grant 362898. We thank the reviewer for their useful comments, which have enhanced the quality of the manuscript.
\section{Preliminaries}
Let us begin by fixing some terminology and conventions. Throughout this paper, by an \textbf{algebra} we mean a commutative,
associative algebra over the field \(\R\) of real numbers. We say that an algebra \(\mathscr A\) is \textbf{unital}
if it has a multiplicative unit. Whenever we say that \(\mathscr A'\) is a unital subalgebra of \(\mathscr A\), we implicitly
make the assumption that \(\mathscr A\) and \(\mathscr A'\) share the same multiplicative unit.
\medskip

Given any real number \(a\in\R\), we denote by \(a^+\coloneqq a\vee 0\) and \(a^-\coloneqq -(a\wedge 0)\) its \textbf{positive part}
and its \textbf{negative part}, respectively. Given a set \(\X\), we indicate by \(\1_S\) the \textbf{characteristic function}
of a subset \(S\) of \(\X\), i.e.\ we define \(\1_S(x)\coloneqq 1\) for every \(x\in S\) and \(\1_S(x)\coloneqq 0\) for every
\(x\in\X\setminus S\).
\medskip

For any two Banach spaces \(\mathbb B_1\) and \(\mathbb B_2\), we denote by \(\mathcal L(\mathbb B_1;\mathbb B_2)\) the vector
space of all bounded linear operators between \(\mathbb B_1\) and \(\mathbb B_2\). Recall that \(\mathcal L(\mathbb B_1;\mathbb B_2)\)
is a Banach space with respect to the operator norm, given by \(\|T\|_{\mathcal L(\mathbb B_1;\mathbb B_2)}\coloneqq
\sup\{\|T(v)\|_{\mathbb B_2}:v\in\mathbb B_1,\,\|v\|_{\mathbb B_1}\leq 1\}\) for all \(T\in\mathcal L(\mathbb B_1;\mathbb B_2)\).
We denote by \(\mathbb B':=\mathcal L(\mathbb B;\R)\) the topological dual of a Banach space \(\mathbb B\).
\subsection{Metric/measure spaces}
In this section, we discuss several concepts and results concerning spaces that are equipped with a distance and/or a measure.
Given a measure space \((\X,\Sigma,\mu)\) and an exponent \(p\in[1,\infty]\), we denote by \((L^p(\mu),\|\cdot\|_{L^p(\mu)})\)
the \textbf{Lebesgue space} of exponent \(p\) over \((\X,\Sigma,\mu)\). We write \(L^p(\mu)^+\) to indicate the set of
all \(f\in L^p(\mu)\) satisfying \(f\geq 0\) \(\mu\)-a.e.\ on \(\X\). We also recall that \(L^p(\mu)\) is a Riesz space
if endowed with the usual pointwise \(\mu\)-a.e.\ order. Under suitable assumptions (for example, if \(\mu\) is \(\sigma\)-finite),
it holds that \(L^p(\mu)\) is a Dedekind complete lattice, which means that every order-bounded subset of \(L^p(\mu)\)
has a supremum and an infimum. For any set \(\rm M\) of non-negative measures on \((\X,\Sigma)\), we denote its
\textbf{supremum measure} by \(\bigvee{\rm M}\):
\[
\bigvee{\rm M}(E)\coloneqq\sup\bigg\{\sum_{n\in\N}\mu_n(E_n)\;\bigg|\;(\mu_n)_{n\in\N}\subseteq{\rm M},
\,(E_n)_{n\in\N}\subseteq\Sigma\text{ partition of }E\bigg\}\quad\text{ for every }E\in\Sigma.
\]
We recall that \(\bigvee{\rm M}\) is the least measure on \((\X,\Sigma)\) satisfying \(\bigvee{\rm M}\geq\mu\)
for every \(\mu\in{\rm M}\), where the partial order \(\leq\) on the set of all non-negative measures on \((\X,\Sigma)\)
is defined in the following way: we declare that \(\mu\leq\nu\) if and only if \(\mu(E)\leq\nu(E)\) for every \(E\in\Sigma\).
\medskip

Let \((\X,\sfd)\) be a given metric space. We denote by \(\rmC(\X)\) the unital algebra (with unit \(\1_\X\)) of all real-valued
continuous functions on \(\X\). We also consider its unital subalgebra \(\rmC_b(\X)\) consisting of all bounded continuous
functions. If \((\mu_n)_{n\in\N}\) and \(\mu\) are finite non-negative Borel measures on \(\X\), then we say that \(\mu_n\)
\textbf{weakly converges} to \(\mu\) if
\[
\int f\,\d\mu=\lim_{n\to\infty}\int f\,\d\mu_n\quad\text{ for every }f\in\rmC_b(\X).
\]We denote by \(\mathcal M(\X)\) the vector space of all finite, signed \textbf{Radon measures} on \(\X\). Then
\(\mathcal M(\X)\) is a Banach space if endowed with the total variation norm.
We denote by \(\mathcal M_+(\X)\) the set of all those measures \(\mu\in\mathcal M(\X)\) that are non-negative.
To any measure \(\mu\in\mathcal M(\X)\), we associate its \textbf{total variation measure} \(|\mu|\in\mathcal M_+(\X)\).
The \textbf{support} of a measure \(\mu\in\mathcal M(\X)\) is the closed subset \({\rm spt}(\mu)\) of \(\X\) given by
\[
{\rm spt}(\mu)\coloneqq\big\{x\in\X\;\big|\;|\mu|(B_r(x))>0\text{ for every }r>0\big\},
\]
where \(B_r(x)\coloneqq\{y\in\X:\sfd(x,y)<r\}\) denotes the open ball in \(\X\) of center \(x\) and radius \(r\).
Since the inner regularity of \(\mu\) ensures that \(\mu\) is concentrated on a \(\sigma\)-compact set, it can be readily checked
that \({\rm spt}(\mu)\) is separable and \(\mu\) is concentrated on \({\rm spt}(\mu)\). With a slight abuse of terminology,
we say that a metric space \((\X,\sfd)\) is \textbf{Radon} if its induced topological space is a Radon space, i.e.\ every
finite Borel measure defined on it is a Radon measure. Since all Souslin spaces are Radon spaces, we have that every (Borel
subset of a) complete and separable metric space is Radon. See e.g.\ \cite[Theorem 7.4.3, Section 7.14(vii)]{Bog:07} and the references therein for
a more detailed discussion.
\medskip

A distinguished subalgebra of \(\rmC(\X)\) is the space of all real-valued \textbf{locally Lipschitz} functions on \(\X\), which is given by
\[
\LIP_{loc}(\X)\coloneqq\bigg\{f\in\rmC(\X)\;\bigg|\;\inf_{r>0}\Lip(f;B_r(x))<+\infty\text{ for every }x\in\X\bigg\},
\]
where the \textbf{Lipschitz constant} \(\Lip(f;E)\) of the function \(f\) on a set \(E\subseteq\X\) is defined as
\[
\Lip(f;E)\coloneqq\sup\bigg\{\frac{|f(x)-f(y)|}{\sfd(x,y)}\;\bigg|\;x,y\in E,\,x\neq y\bigg\}\in[0,+\infty],
\]
where we adopt the convention that \(\sup(\varnothing)\coloneqq 0\). The \textbf{asymptotic slope}
\(\lip_a(f)\colon\X\to[0,+\infty)\) of a function \(f\in\LIP_{loc}(\X)\) is given by
\begin{equation}\label{eq:def_lip_a}
\lip_a(f)(x)\coloneqq\inf_{r>0}\Lip(f;B_r(x))\quad\text{ for every }x\in\X.
\end{equation}
It is easy to check that \(\lip_a(f)\) is upper semicontinuous (thus, Borel measurable). The space of all
(globally) Lipschitz functions, i.e.\ of those \(f\in\LIP_{loc}(\X)\) with \(\Lip(f)\coloneqq\Lip(f;\X)<+\infty\), is
denoted by \(\LIP(\X)\). Note that \(\LIP_b(\X)\coloneqq\rmC_b(\X)\cap\LIP(\X)\) is a unital subalgebra of \(\rmC_b(\X)\).
\medskip

In this paper, we focus on the following class of metric measure spaces:
\begin{definition}[Metric measure space]\label{def:mms}
We say that a triple \((\X,\sfd,\mm)\) is a \textbf{metric measure space} if \((\X,\sfd)\) is a metric space
and \(\mm\in\mathcal M_+(\X)\) (i.e.\ \(\mm\geq 0\) is a finite Radon measure on \(\X\)).
\end{definition}

We point out that this notion of metric measure space is different from other approaches considered
in the literature. In this paper, the following subalgebra of \(\rmC_b(\X)\) will have a key role.
\begin{definition}[The algebra \(\LIP_\star(\X)\)]\label{def:algebra_LIP_star}
Let \((\X,\sfd,\mm)\) be a metric measure space. Then we define
\[
\LIP_\star(\X)\coloneqq\big\{f\in\rmC_b(\X)\cap\LIP_{loc}(\X)\;\big|\;\lip_a(f)\in L^1(\mm)\big\}.
\]
\end{definition}

Note that \(\LIP_\star(\X)\) depends both on the distance \(\sfd\) and the measure \(\mm\), even though
we do not indicate it for brevity. Since \(\LIP_\star(\X)\) is a vector subspace of \(\rmC_b(\X)\) and
\(fg\in\LIP_\star(\X)\) for every \(f,g\in\LIP_\star(\X)\) (as \(\lip_a(fg)\leq|f|\lip_a(g)+|g|\lip_a(f)\in L^1(\mm)\)),
we have that \(\LIP_\star(\X)\) is a unital subalgebra of \(\rmC_b(\X)\). We will work also with unital \textbf{separating}
subalgebras \(\mathscr A\) of \(\LIP_\star(\X)\), i.e.
\[
\sup_{f\in\mathscr A}|f(x)-f(y)|>0\quad\text{ for every }x,y\in\X\text{ with }x\neq y.
\]
By the Stone--Weierstrass theorem, if \(\mathscr A\) is a unital separating subalgebra of \(\LIP_\star(\X)\), then
\begin{equation}\label{eq:Stone-Weiestrass}
\mathscr A\text{ is dense in }L^p(\mm)\text{ for every }p\in[1,\infty),
\end{equation}
see e.g.\ \cite[Lemma 2.1.27]{Sav:22} (notice that, even if there the algebra $\mathscr A$ is assumed to be \emph{compatible}, one can easily check that the simple proof of the mentioned Lemma does not require additional assumptions on $\mathscr A$).  It can be readily checked that \(\LIP_\star(\X)\) itself is a unital separating algebra. It is essentially because of the result in \eqref{eq:Stone-Weiestrass} that in the paper we will always assume that the algebras we are working with are unital and separating: both assumptions are indeed necessary to prove density results, which ultimately rely on the Stone--Weierstrass theorem.
Furthermore, the fact that the algebras \(\mathscr A\) that we consider contain all constant functions (as follows from the assumption that \(\mathscr A\) is
unital) will be implicitly used throughout the paper in several truncation arguments (for example, in the proof of Lemma \ref{lem:part_of_unity}). Indeed, when we
apply \cite[Corollary 2.1.24]{Sav:22} to approximate a function of the form \(f\wedge\lambda\), with \(f\in\mathscr A\) and \(\lambda\in\R\), with another function in \(\mathscr A\)
we rely on the fact that \(\lambda\1_X\in\mathscr A\). As we pointed out in the Introduction, allowing for constant functions in \(\mathscr A\), we are forced to
consider only finite reference measures \(\mm\) (to make sure that all functions in \(\mathscr A\) are integrable). It would be interesting, but outside the scopes of this paper
and likely quite technical, to generalise some of the presented results to suitable (possibly non-unital) algebras and non-finite reference measures.
\medskip

In the sequel, we will need a partition-of-unity-type result involving functions of a given unital separating subalgebra of \(\LIP_\star(\X)\),
see Lemma \ref{lem:part_of_unity} below. To prove it, we will use the following simple consequence of Weierstrass' approximation theorem:
\begin{lemma}\label{lem:weier}
Let $a,b \in \R$ with $a<b$ and let $f \in \rmC([a,b])$ be such that $f(t) > 0$ for every $t \in [a,b]$. Then there exists a sequence $(p_n)_n$ of polynomials such that $0 \le p_n(t) \le f(t)$ for every $t \in [a,b]$ and every $n \in \N$, and $p_n$ converges uniformly on $[a,b]$ to $f$ as $n \to \infty$.
\end{lemma}
\begin{proof} Let $\alpha:= \min_{[a,b]} f>0$. For every $n \in \N$ with $n>n_0\coloneqq\lceil 2/\alpha \rceil$, by the Weierstrass
approximation theorem, there is a polynomial $\tilde p_n$ such that $\|f-\tilde p_n\|_{\rmC_b([a,b])}<1/n$. It is then enough
to set $p_n\coloneqq\tilde p_{n_0+n} -\frac{1}{n+n_0}$ for every $n \in \N$.
\end{proof}
\begin{lemma}\label{lem:part_of_unity}
Let \((\X,\sfd,\mm)\) be a metric measure space, \(\mathscr A\) a unital separating subalgebra of \(\LIP_\star(\X)\) and
\(E_1,\ldots,E_n\) a Borel partition of \(\X\). Fix any \(\varepsilon>0\).
Then there exist \(\eta_1,\ldots,\eta_n\in\mathscr A\) such that the following properties hold:
\begin{itemize}
\item[\(\rm i)\)] \(0\leq\eta_i\leq 1\) for every \(i=1,\ldots,n\).
\item[\(\rm ii)\)] \(\sum_{i=1}^n\eta_i(x)\leq 1\) for every \(x\in\X\).
\item[\(\rm iii)\)] \(\|\eta_i-\1_{E_i}\|_{L^1(\mm)}\leq\varepsilon\) for every \(i=1,\ldots,n\).
\end{itemize}
\end{lemma}
\begin{proof}
Fix \(\delta>0\). Since \(\mathscr A\) is dense in \(L^1(\mm)\), for any \(i=1,\ldots,n\) we can find
\((\psi^k_i)_k\subseteq\mathscr A\) such that \(\1_{E_i}(x)=\lim_k\psi^k_i(x)\) for \(\mm\)-a.e.\ \(x\in\X\).
Thanks to \cite[Corollary 2.1.24]{Sav:22}, we can also assume that \(0\leq\psi^k_i\leq 1\). By Lemma \ref{lem:weier}, there is a polynomial \(p\) such that
\[
0\leq p(t)\leq\frac{1}{t},\qquad\bigg|p(t)-\frac{1}{t}\bigg|\leq\delta\quad\text{ for every }t\in[\delta,n+\delta].
\]
Since \(p\big(\sum_{j=1}^n\psi^k_j(x)+\delta\big)\to p(1+\delta)\) as \(k\to\infty\) for \(\mm\)-a.e.\ \(x\in\X\)
and \(p\circ\big(\sum_{j=1}^n\psi^k_j+\delta\big)\leq\frac{1}{\delta}\) holds \(\mm\)-a.e.\ on \(\X\)
for every \(k\in\N\), by the dominated convergence theorem we can find \(k_0\in\N\) such that, letting
\(\psi_i\coloneqq\psi^{k_0}_i\), it holds that
\(\big\|p\big(\sum_{j=1}^n\psi_j+\delta\big)-p(1+\delta)\big\|_{L^1(\mm)}\leq\delta\)
and \(\|\psi_i-\1_{E_i}\|_{L^1(\mm)}\leq\delta\) for every \(i=1,\ldots,n\). Now, let us define
\[
\eta_i\coloneqq p\bigg(\sum_{j=1}^n\psi_j+\delta\bigg)\psi_i\in\mathscr A\quad\text{ for every }i=1,\ldots,n.
\]
Notice that \(\eta_i\geq 0\) for every \(i=1,\ldots,n\) and \(\sum_{i=1}^n\eta_i\leq\big(\sum_{j=1}^n\psi_j+\delta\big)^{-1}\big(\sum_{i=1}^n\psi_i\big)\leq 1\),
which proves i) and ii). Finally, for any given \(i=1,\ldots,n\) we can estimate
\[\begin{split}
\|\eta_i-\1_{E_i}\|_{L^1(\mm)}&\leq\bigg\|p\bigg(\sum_{j=1}^n\psi_j+\delta\bigg)\psi_i-\psi_i\bigg\|_{L^1(\mm)}+\|\psi_i-\1_{E_i}\|_{L^1(\mm)}\\
&\leq\bigg\|p\bigg(\sum_{j=1}^n\psi_j+\delta\bigg)-\frac{1}{1+\delta}\bigg\|_{L^1(\mm)}+\bigg(1-\frac{1}{1+\delta}\bigg)\mm(\X)+\delta\\
&\leq\bigg\|p\bigg(\sum_{j=1}^n\psi_j+\delta\bigg)-p(1+\delta)\bigg\|_{L^1(\mm)}
+\bigg|p(1+\delta)-\frac{1}{1+\delta}\bigg|\mm(\X)+(\mm(\X)+1)\delta\\
&\leq 2(\mm(\X)+1)\delta,
\end{split}\]
which proves the validity of iii) if we choose \(\delta>0\) so that \(2(\mm(\X)+1)\delta\leq\varepsilon\).
\end{proof}
\subsection{\texorpdfstring{\(L^p\)}{Lp}-Banach \texorpdfstring{\(L^\infty\)}{Linfty}-modules}
Let us recall the language of \emph{\(L^p\)-Banach \(L^\infty\)-modules} that was introduced in \cite{Gig:18} (see also \cite{Gig:17}).
Following \cite[Definition 1.2.10]{Gig:18} (and with a slight change of terminology), we give the following definition:
\begin{definition}[\(L^p\)-Banach \(L^\infty\)-module]
Let \((\X,\Sigma,\mu)\) be a \(\sigma\)-finite measure space and \(p\in[1,\infty]\). Let \(\mathscr M\) be a module over the commutative ring \(L^\infty(\mu)\).
Then we say that \(\mathscr M\) is an \textbf{\(L^p(\mu)\)-Banach \(L^\infty(\mu)\)-module} if it is endowed with a map \(|\cdot|\colon\mathscr M\to L^p(\mu)^+\),
called \textbf{pointwise norm}, such that the following conditions are satisfied:
\begin{itemize}
\item[\(\rm i)\)] Given any \(v,w\in\mathscr M\) and \(f\in L^\infty(\mu)\), it holds that
\[\begin{split}
&|v|=0\quad\Longrightarrow\quad v=0,\\
&|v+w|\leq|v|+|w|,\\
&|fv|=|f||v|,
\end{split}\]
where all equalities and inequalities involving the pointwise norm are in the \(\mu\)-a.e.\ sense.
\item[\(\rm ii)\)] The \textbf{glueing property} holds, i.e.\ if \((E_n)_{n\in\N}\subseteq\Sigma\) is a partition of the set \(\X\) and
a sequence \((v_n)_{n\in\N}\subseteq\mathscr M\) satisfies \(\sum_{n\in\N}\|\1_{E_n}|v_n|\|_{L^p(\mu)}^p<+\infty\), then there exists
a (necessarily unique) element \(v=\sum_{n\in\N}\1_{E_n}v_n\in\mathscr M\) such that \(\1_{E_n}v=\1_{E_n}v_n\) for every \(n\in\N\).
\item[\(\rm iii)\)] The norm \(\|v\|_{\mathscr M}\coloneqq\||v|\|_{L^p(\mu)}\) on \(\mathscr M\) is complete.
\end{itemize}
\end{definition}

Observe that each \(L^p(\mu)\)-Banach \(L^\infty(\mu)\)-module is in particular a Banach space. We say that a subset \(S\) of \(\mathscr M\)
\textbf{generates} \(\mathscr M\) provided the \(L^\infty(\mu)\)-linear span of \(S\) is dense in \(\mathscr M\).
An \textbf{isomorphism} \(\Phi\colon\mathscr M\to\mathscr N\) between two \(L^p(\mu)\)-Banach \(L^\infty(\mu)\)-modules \(\mathscr M\)
and \(\mathscr N\) is an \(L^\infty(\mu)\)-linear bijection satisfying \(|\Phi(v)|=|v|\) for every \(v\in\mathscr M\).
\medskip

Since any given \(L^1(\mu)\)-Banach \(L^\infty(\mu)\)-module \(\mathscr M\) is a Banach space, we can consider its dual Banach space \(\mathscr M'\).
Alternatively, we can consider its dual in the sense of Banach modules: following \cite[Definition 1.2.6]{Gig:18}, we define \(\mathscr M^*\) as the set
of all \(L^\infty(\mu)\)-linear maps \(\omega\colon\mathscr M\to L^1(\mu)\) for which there exists a function \(g\in L^\infty(\mu)^+\) such that
\begin{equation}\label{eq:def_dual_mod}
|\omega(v)|\leq g|v|\text{ holds }\mu\text{-a.e.\ on }\X\text{ for every }v\in\mathscr M.
\end{equation}
Then \(\mathscr M^*\) is an \(L^\infty(\mu)\)-Banach \(L^\infty(\mu)\)-module if endowed with the pointwise operations and with
\[
|\omega|\coloneqq\bigwedge\big\{g\in L^\infty(\mu)^+\;\big|\;g\text{ satisfies \eqref{eq:def_dual_mod}}\big\}\in L^\infty(\mu)^+\quad\text{ for every }\omega\in\mathscr M^*.
\]
We say that \(\mathscr M^*\) is the \textbf{module dual} of \(\mathscr M\). The relation between \(\mathscr M'\) and \(\mathscr M^*\) is
clarified by the following result: letting \(\textsc{Int}_{\mathscr M}\colon\mathscr M^*\to\mathscr M'\) be the operator given by
\[
\textsc{Int}_{\mathscr M}(\omega)(v)\coloneqq\int\omega(v)\,\d \mu\quad\text{ for every }\omega\in\mathscr M^*\text{ and }v\in\mathscr M,
\]
we know from \cite[Proposition 1.2.13]{Gig:18} that
\begin{equation}\label{eq:Int}
\textsc{Int}_{\mathscr M}\text{ is an isometric isomorphism of Banach spaces.}
\end{equation}
\subsection{Lipschitz derivations}
In this paper, we consider \emph{Lipschitz derivations} in the sense of \cite{DiMar:14,DiMaPhD:14}.
Borrowing from \cite[Section 1.1]{DiMar:14} and \cite[Section 4.1]{Amb:Iko:Luc:Pas:24}, we give the ensuing definitions:
\begin{definition}[Derivation]\label{def:der}
Let \((\X,\sfd,\mm)\) be a metric measure space. Let \(\mathscr A\) be a unital separating subalgebra of \(\LIP_\star(\X)\).
Then a map \(b\colon\mathscr A\to L^1(\mm)\) is called a \textbf{Lipschitz \(\mathscr A\)-derivation} if:
\begin{itemize}
\item[\(\rm i)\)] There exists \(G\in L^\infty(\mm)^+\) such that \(|b(f)|\leq G\,\lip_a(f)\) holds \(\mm\)-a.e.\ for every \(f\in\mathscr A\).
\item[\(\rm ii)\)] The map \(b\) is a linear operator that satisfies the \textbf{Leibniz rule}, i.e.
\[
b(fg)=f\,b(g)+g\,b(f)\quad\text{ for every }f,g\in\mathscr A.
\]
\end{itemize}
We denote by \(\Der^\infty(\X;\mathscr A)\) the space of all Lipschitz \(\mathscr A\)-derivations on \((\X,\sfd,\mm)\).
\end{definition}

The space \(\Der^\infty(\X;\mathscr A)\) is a module over the ring \(L^\infty(\mm)\), thus in particular it is a vector space.
Given any \(b\in\Der^\infty(\X;\mathscr A)\), we denote by \(|b|_{\mathscr A}\in L^\infty(\mm)^+\) the \(\mm\)-a.e.\ minimal
function \(G\) as in Definition \ref{def:der} i), whose existence is guaranteed by the Dedekind completeness of \(L^\infty(\mm)\).
The space \((\Der^\infty(\X;\mathscr A),|\cdot|_{\mathscr A})\) is an \(L^\infty(\mm)\)-Banach \(L^\infty(\mm)\)-module.
When \(\mathscr A=\LIP_\star(\X)\), we just write \(\Der^\infty(\X)\) and \(|b|\) for brevity.
\begin{definition}[Divergence]\label{def:div}
Let \((\X,\sfd,\mm)\) be a metric measure space. Let \(\mathscr A\) be a unital separating subalgebra of \(\LIP_\star(\X)\).
Then we say that \(b\in\Der^\infty(\X;\mathscr A)\) has \textbf{divergence} \(\div(b)\in L^\infty(\mm)\) if
\[
\int b(f)\,\d\mm=-\int f\,\div(b)\,\d\mm\quad\text{ for every }f\in\mathscr A.
\]
We denote by \(\Der^\infty_\infty(\X;\mathscr A)\) the space of all derivations \(b\in\Der^\infty(\X;\mathscr A)\) having divergence.
\end{definition}
\begin{remark}[Independence of \(\div(b)\) from $\mathscr A$]{\rm
As \(\mathscr A\) is dense in \(L^1(\mm)\) by \eqref{eq:Stone-Weiestrass}, we have that \(\div(b)\) is uniquely determined.
In principle, \(\div(b)\) may depend on the choice of the algebra \(\mathscr A\). However, it is easy to check
that if \(\mathscr A\subseteq\tilde{\mathscr A}\) are unital separating subalgebras of \(\LIP_\star(\X)\) and
\(b\in\Der^\infty_\infty(\X;\tilde{\mathscr A})\), then \(b\in\Der^\infty_\infty(\X;\mathscr A)\) and
the two notions of divergence (in duality with \(\mathscr A\) and \(\tilde{\mathscr A}\), respectively) coincide.
Thus, with a slight abuse of notation, we will write only \(\div(b)\) without stressing the dependence on the chosen
algebra, since it will be clear from the context.
\fr}\end{remark}

Note also that the space \(\Der^\infty_\infty(\X;\mathscr A)\) is a module over the ring \(\mathscr A\),
thus in particular it is a vector subspace of \(\Der^\infty(\X;\mathscr A)\). We denote by
\[
\overline\Der^\infty_\infty(\X;\mathscr A)
\]
the closure of \(\Der^\infty_\infty(\X;\mathscr A)\) in \(\Der^\infty(\X;\mathscr A)\). Note that \(\overline\Der^\infty_\infty(\X;\mathscr A)\)
is a Banach space with respect to the norm induced by \(\Der^\infty(\X;\mathscr A)\).
\section{Metric \texorpdfstring{\(\rm BV\)}{BV} and \texorpdfstring{\({\rm H}^{1,1}\)}{H11} spaces}
\subsection{The space \texorpdfstring{\({\rm BV}_{\rm H}\)}{BVH}}
A metric BV space defined in terms of approximating (locally) Lipschitz functions was introduced in \cite{Mir:03}
and revisited in \cite{Amb:DiMa:14}. Here, we consider its following variant:
\begin{definition}[BV space via relaxation]\label{def:bvh}
Let \((\X,\sfd,\mm)\) be a metric measure space. Let \(\mathscr A\) be a unital separating subalgebra of \(\LIP_\star(\X)\).
Then we define \({\rm BV}_{\rm H}(\X;\mathscr A)\) as the space of all \(f\in L^1(\mm)\) for which there exists
\((f_n)_n\subseteq\mathscr A\) such that \(f_n\to f\) in \(L^1(\mm)\) and \(\sup_{n\in\N}\int\lip_a(f_n)\,\d\mm<+\infty\).
\end{definition}

The \textbf{total variation} of a function \(f\in{\rm BV}_{\rm H}(\X;\mathscr A)\) is given by the quantity
\begin{equation}\label{eq:bvh_tv}
\|{\bf D}f\|_{*,\mathscr A}\coloneqq\inf\bigg\{\liminf_{n\to\infty}\int\lip_a(f_n)\,\d\mm\;\bigg|\;
(f_n)_n\subseteq\mathscr A,\,f_n\to f\text{ in }L^1(\mm)\bigg\}<+\infty.
\end{equation}
Clearly, there is a sequence \((f_n)_n\subseteq\mathscr A\) such that \(f_n\to f\) in \(L^1(\mm)\) and
\(\int\lip_a(f_n)\,\d\mm\to\|{\bf D}f\|_{*,\mathscr A}\). The \textbf{Cheeger \(\mathscr A\)-energy} functional
\(\mathcal E_{\mathscr A}\colon L^1(\mm)\to[0,+\infty]\) is then defined as
\[
\mathcal E_{\mathscr A}(f)\coloneqq\left\{\begin{array}{ll}
\|{\bf D}f\|_{*,\mathscr A}\\
+\infty
\end{array}\quad\begin{array}{ll}
\text{ if }f\in{\rm BV}_{\rm H}(\X;\mathscr A),\\
\text{ otherwise.}
\end{array}\right.
\]
It is a direct consequence of the definition that
\begin{equation}\label{eq:incl_BV_H}
{\rm BV}_{\rm H}(\X;\mathscr A)\subseteq{\rm BV}_{\rm H}(\X;\tilde{\mathscr A}),\qquad\mathcal E_{\tilde{\mathscr A}}\leq\mathcal E_{\mathscr A}
\quad\text{ whenever }\mathscr A\subseteq\tilde{\mathscr A}.
\end{equation}
In the case \(\mathscr A=\LIP_\star(\X)\), we use the shorthand notations \({\rm BV}_{\rm H}(\X)\) and \(\|{\bf D}f\|_*\).
To any given function in \({\rm BV}_{\rm H}(\X)\), one can associate also a total variation measure, as follows:
\begin{definition}[Total variation measure \(|{\bf D}f|_*\)]\label{def:BV_H_meas}
Let \((\X,\sfd,\mm)\) be a metric measure space. Let \(f\in{\rm BV}_{\rm H}(\X)\) be given. Then for any open set \(U\subseteq\X\)
we define
\[
|{\bf D}f|_*(U)\coloneqq\inf\bigg\{\liminf_{n\to\infty}\int_U\lip_a(f_n)\,\d\mm\;\bigg|\;\LIP_\star(U)\ni f_n\to f|_U\text{ in }L^1(\mm|_U)\bigg\}\in[0,\|{\bf D}f\|_*].
\]
\end{definition}

Observe that \(|{\bf D}f|_*(\X)=\|{\bf D}f\|_*\). Furthermore, by arguing as in \cite[Lemma 5.2]{Amb:DiMa:14} one can show that the set-function introduced in
Definition \ref{def:BV_H_meas} induces via Carath\'{e}odory's construction a finite Borel measure, still denoted by \(|{\bf D}f|_*\). Videlicet,
the Borel measure \(|{\bf D}f|_*\) is characterised by
\[
|{\bf D}f|_*(E)\coloneqq\inf\big\{|{\bf D}f|_*(U)\;\big|\;U\subseteq\X\text{ open, }E\subseteq U\big\}\quad\text{ for every Borel set }E\subseteq\X.
\]
This is shown in Appendix \ref{app:measure}, explaining why the arguments in the proofs of \cite[Lemmata 5.2 and 5.4]{Amb:DiMa:14} do not make use of the completeness nor separability
of the ambient metric space. Moreover, even though in \cite[Eq.\ (5.1)]{Amb:DiMa:14} the approximating sequence in the definition of \(|{\bf D}f|_*(U)\) is
not required to consist of bounded functions, it still coincides with our definition thanks to a standard truncation argument.
\begin{remark}\label{rmk:Portmanteau}{\rm
Assume \(f\in{\rm BV}_{\rm H}(\X)\) and \((f_n)_n\subseteq\LIP_\star(\X)\) are chosen so that \(f_n\to f\) in \(L^1(\mm)\) and \(\int\lip_a(f_n)\,\d\mm\to\|{\bf D}f\|_*\).
Then it holds that
\[
\lip_a(f_n)\mm\rightharpoonup|{\bf D}f|_*\quad\text{ weakly.}
\]
Indeed, for any \(U\subseteq\X\) open we have \(|{\bf D}f|_*(U)\leq\liminf_n\int_U\lip_a(f_n)\,\d\mm=\liminf_n(\lip_a(f_n)\mm)(U)\).
Taking into account also that \((\lip_a(f_n)\mm)(\X)\to|{\bf D}f|_*(\X)\), we conclude that \(\lip_a(f_n)\mm\rightharpoonup|{\bf D}f|_*\) weakly,
thanks to the Portmanteau theorem.
\fr}\end{remark}
\begin{remark}\label{rmk:|Df|_*_Radon}{\rm
It follows from Remark \ref{rmk:Portmanteau} that
\[
|{\bf D}f|_*\text{ is concentrated on }{\rm spt}(\mm)\text{ for every }f\in{\rm BV}_{\rm H}(\X).
\]
Indeed, as \(\X\setminus{\rm spt}(\mm)\) is open, we have that \(|{\bf D}f|_*(\X\setminus{\rm spt}(\mm))\leq\liminf_n\int_{\X\setminus{\rm spt}(\mm)}\lip_a(f_n)\,\d\mm=0\).
However, even though \({\rm spt}(\mm)\) is separable, it could happen that \(|{\bf D}f|_*\) is not a Radon measure, but under the additional assumption that \((\X,\sfd)\)
is complete it holds that \(|{\bf D}f|_*\) is Radon (since in this case \({\rm spt}(\mm)\) is complete and separable, thus it is a Radon space).
\fr}\end{remark}

It is highly non-trivial to prove that, assuming the ambient metric space is complete, every function in \({\rm BV}_{\rm H}(\X)\) can be approximated
(on the whole \(\X\)) as in Definition \ref{def:bvh} by a sequence of bounded Lipschitz functions. The validity of this claim follows from the results of \cite{DiMaPhD:14}, as we are going to see:
\begin{proposition}\label{prop:approx_with_glob_Lip}
Let \((\X,\sfd,\mm)\) be a metric measure space such that \((\X,\sfd)\) is complete. Then
\[
{\rm BV}_{\rm H}(\X)={\rm BV}_{\rm H}(\X;\LIP_b(\X))
\]
and it holds that \(\|{\bf D}f\|_{*,\LIP_b(\X)}=\|{\bf D}f\|_*\) for every \(f\in{\rm BV}_{\rm H}(\X)\).
\end{proposition}
\begin{proof}
Under the additional assumption that \((\X,\sfd)\) is separable, the statement is a direct consequence of \cite[Theorem 4.5.3]{DiMaPhD:14}. Let us now show how
to drop the separability assumption. For brevity, let us denote by \(S\) the support of \(\mm\). Defining the so-called \emph{weak-BV spaces} \({\rm BV}_w\) as
it is done in \cite[Section 4.4.3]{DiMaPhD:14} (or \cite[Section 5.3]{Amb:DiMa:14}), by arguing as in \cite[Section 4.4]{DiMaPhD:14} one can check that
\[
{\rm BV}_{\rm H}(\X;\LIP_b(\X))\subseteq{\rm BV}_{\rm H}(\X)\subseteq{\rm BV}_w(\X),
\]
as well as \(|{\bf D}f|_w(\X)\leq\|{\bf D}f\|_*\leq\|{\bf D}f\|_{*,\LIP_b(\X)}\) for every \(f\in{\rm BV}_{\rm H}(\X;\LIP_b(\X))\), where \(|{\bf D}f|_w\) denotes
the total variation measure corresponding to the weak-BV space \({\rm BV}_w(\X)\). Observe that \({\rm BV}_{\rm H}(\X;\LIP_b(\X))\) can be identified with
\({\rm BV}_{\rm H}(S;\LIP_b(S))\) (thanks to \cite{DiMa:Gig:Pra:20}), and that \({\rm BV}_w(\X)\) can be identified with \({\rm BV}_w(S)\)
(since \(\infty\)-test plans are concentrated on curves that lie in the support of \(\mm\)); in both cases, also the total variations coincide. Since
\({\rm BV}_{\rm H}(S;\LIP_b(S))={\rm BV}_w(S)\) and \(|{\bf D}f|_w(S)=\|{\bf D}f\|_{*,\LIP_b(S)}\) for every \(f\in{\rm BV}_w(S)\) by \cite[Theorem 4.5.3]{DiMaPhD:14}
(together with the fact that \(S\) is complete and separable), we can finally conclude that the statement holds.
\end{proof}

We point out that the completeness assumption in Proposition \ref{prop:approx_with_glob_Lip} cannot be dropped, as it is shown by the following example:
\begin{example}\label{ex:no_glob_approx}{\rm
Let us consider the space \(\X\subseteq\R\), which we define as
\[
\X\coloneqq\bigcup_{n=1}^\infty(n-2^{-n},n+2^{-n})\setminus\{n\}.
\]
We equip \(\X\) with the restriction \(\sfd\) of the Euclidean distance and the restriction \(\mm\) of the one-dimensional Lebesgue measure.
Note that \((\X,\sfd,\mm)\) is a metric measure space with \((\X,\sfd)\) separable but non-complete. Let us now consider the function
\[
f\coloneqq\1_{\bigcup_{n=1}^\infty(n,n+2^{-n})}\in  \LIP_\star(\X).
\]
Since $\lip_a(f)\equiv 0$ we see that \(f\in{\rm BV}_{\rm H}(\X)\) with \(\|{\bf D}f\|_*=0\). However, we claim that \(f\notin{\rm BV}_{\rm H}(\X;\LIP_b(\X))\). To prove it, fix an arbitrary sequence \((f_j)_j\subseteq\LIP_b(\X)\) such that
\(f_j\to f\) in \(L^1(\mm)\). Note that each \(f_j\) can be uniquely extended to a bounded Lipschitz function on \(\bigcup_{n=1}^\infty[n-2^{-n},n+2^{-n}]\),
which we still denote by \(f_j\). Up to a passing to a non-relabelled subsequence, we can assume that \(\liminf_j\int\lip_a(f_j)\,\d\mm\) is in fact
a limit. Up to a further subsequence, we also have that \(f_j\to f\) in the pointwise \(\mm\)-a.e.\ sense. In particular, for any given \(k\in\N\) we can find
\(j_k\in\N\) and \((t^k_\ell)_{\ell=1}^k,(s^k_\ell)_{\ell=1}^k\subseteq\X\), with \(\ell-2^{-\ell}<t^k_\ell<\ell<s^k_\ell<\ell+2^{-\ell}\) for all \(\ell=1,\ldots,k\),
such that \(f_j(t^k_\ell)\leq\frac{1}{3}\) and \(f_j(s^k_\ell)\geq\frac{2}{3}\) for every \(\ell=1,\ldots,k\) and \(j\geq j_k\). Then
\[
\frac{1}{3}\leq f_j(s^k_\ell)-f_j(t^k_\ell)=\int_{t^k_\ell}^{s^k_\ell}f'_j(r)\,\d r\leq\int_{\ell-2^{-\ell}}^{\ell+2^{-\ell}}\lip_a(f_j)(r)\,\d r
\quad\text{ for all }\ell=1,\ldots,k\text{ and }j\geq j_k.
\]
Summing over \(\ell=1,\ldots,k\) and letting \(j\to\infty\), we deduce that
\[
\lim_{j\to\infty}\int\lip_a(f_j)\,\d\mm\geq\limsup_{j\to\infty}\sum_{\ell=1}^k\int_{\ell-2^{-\ell}}^{\ell+2^{-\ell}}\lip_a(f_j)(r)\,\d r\geq\frac{k}{3}.
\]
Due to the arbitrariness of \(k\in\N\), we conclude that \(\lim_j\int\lip_a(f_j)\,\d\mm=+\infty\) and thus accordingly \(f\notin{\rm BV}_{\rm H}(\X;\LIP_b(\X))\).
}
\fr\end{example}
\subsection{The space \texorpdfstring{\({\rm H}^{1,1}\)}{H11}}
A metric Sobolev space defined in terms of approximating Lipschitz functions was introduced in \cite{Ch:99}
and revisited in \cite{Amb:Gig:Sav:13}. Here, we consider its following variant:
\begin{definition}[The space \({\rm H}^{1,1}\)]\label{def:H11}
Let \((\X,\sfd,\mm)\) be a metric measure space. Let \(\mathscr A\) be a unital separating subalgebra of
\(\LIP_\star(\X)\). Then we define \({\rm H}^{1,1}(\X;\mathscr A)\) as the space of all \(f\in L^1(\mm)\)
for which there exist a sequence \((f_n)_n\subseteq\mathscr A\) and a function \(G\in L^1(\mm)^+\) such that
\(f_n\to f\) strongly in \(L^1(\mm)\) and \(\lip_a(f_n)\rightharpoonup G\) weakly in \(L^1(\mm)\).
\end{definition}

Given any \(f\in{\rm H}^{1,1}(\X;\mathscr A)\), we say that \(\tilde G\in L^1(\mm)^+\) is a \textbf{relaxed \(\mathscr A\)-slope} of \(f\)
if \(\tilde G\geq G\) for some \(G\in L^1(\mm)^+\) as in Definition \ref{def:H11}. The set of all relaxed \(\mathscr A\)-slopes
of \(f\) is a closed convex sublattice of \(L^1(\mm)\), thus it admits a unique minimal element \(|{\rm D}f|_{*,\mathscr A}\),
which we call the \textbf{minimal relaxed \(\mathscr A\)-slope}
of \(f\).
The space \({\rm H}^{1,1}(\X;\mathscr A)\) is a Banach space if endowed with the norm
\[
\|f\|_{{\rm H}^{1,1}(\X;\mathscr A)}\coloneqq\|f\|_{L^1(\mm)}+\||{\rm D}f|_{*,\mathscr A}\|_{L^1(\mm)}\quad\text{ for every }f\in{\rm H}^{1,1}(\X;\mathscr A).
\]
Observe that \(\mathscr A\subseteq{\rm H}^{1,1}(\X;\mathscr A)\) and
\begin{equation}\label{eq:LIP_in_H11}
|{\rm D}f|_{*,\mathscr A}\leq\lip_a(f)\quad\text{ for every }f\in\mathscr A.
\end{equation}
Minimal relaxed \(\mathscr A\)-slopes satisfy also the following locality property: for any \(f,g\in{\rm H}^{1,1}(\X;\mathscr A)\),
\begin{equation}\label{eq:locality_rel_slope}
|{\rm D}f|_{*,\mathscr A}=|{\rm D}g|_{*,\mathscr A}\quad\text{ holds }\mm\text{-a.e.\ on }\{f=g\}.
\end{equation}
The interested reader can find a proof of the existence of the minimal relaxed $\mathscr A$-slope and of the completeness of $H^{1,p}(\X; \mathscr{A})$ in \cite[Section 5.2]{Amb:Iko:Luc:Pas:24} for $\mathscr A$ being the algebra of Lipschitz functions with bounded support, or in \cite[Section 3.1]{Sav:22} in case $p>1$. The locality property of the relaxed $\mathscr A$-slope can be proven exactly as in \cite[Theorem 3.1.12(b)]{Sav:22}, where it is proven for $p>1$.

In the case where \(\mathscr A=\LIP_\star(\X)\), we use the shorthand notations \({\rm H}^{1,1}(\X)\) and \(|{\rm D}f|_*\). 
\begin{remark}\label{rmk:BV_vs_H11}{\rm
It holds that \({\rm H}^{1,1}(\X;\mathscr A)\subseteq{\rm BV}_{\rm H}(\X;\mathscr A)\) and
\[
\|{\bf D}f\|_{*,\mathscr A}\leq  \int |{\rm D}f|_{*,\mathscr A}\,\d\mm\quad
\text{ for every }f\in{\rm H}^{1,1}(\X;\mathscr A).
\]
The proof follows immediately from the definitions of the spaces and the weak lower semicontinuity of the $L^1(\mm)$-norm.
\fr}\end{remark}

In accordance with \cite[Definition 2.2.1]{Gig:18}, the metric Sobolev space \({\rm H}^{1,1}(\X;\mathscr A)\) induces a space
of `integrable \(1\)-forms', which is an \(L^1(\mm)\)-Banach \(L^\infty(\mm)\)-module called the \emph{cotangent module}, and
a \emph{differential} map that underlies the minimal relaxed slopes. Videlicet:
\begin{theorem}[Cotangent module]
Let \((\X,\sfd,\mm)\) be a metric measure space. Let \(\mathscr A\) be a unital separating subalgebra of \(\LIP_\star(\X)\).
Then there exist an \(L^1(\mm)\)-Banach \(L^\infty(\mm)\)-module \(L^1(T^*\X;\mathscr A)\), which we call the
\textbf{\(\mathscr A\)-cotangent module}, and a linear operator \(\d_{\mathscr A}\colon\mathscr A\to L^1(T^*\X;\mathscr A)\) such that:
\begin{itemize}
\item[\(\rm i)\)] \(|\d_{\mathscr A}f|=|{\rm D}f|_{*,\mathscr A}\) for every \(f\in\mathscr A\).
\item[\(\rm ii)\)] \(\{\d_{\mathscr A}f:f\in\mathscr A\}\) generates \(L^1(T^*\X;\mathscr A)\).
\end{itemize}
The pair \((L^1(T^*\X;\mathscr A),\d_{\mathscr A})\) is unique up to a unique isomorphism:
given any pair \((\mathscr M,\d)\) with the same properties, there exists a unique isomorphism
of \(L^1(\mm)\)-Banach \(L^\infty(\mm)\)-modules \(\Phi\colon L^1(T^*\X;\mathscr A)\to\mathscr M\)
such that \(\d=\Phi\circ\d_{\mathscr A}\). Moreover, \(\d_{\mathscr A}\) satisfies the \textbf{Leibniz rule}, i.e.
\begin{equation}\label{eq:Leibniz_d_A}
\d_{\mathscr A}(fg)=f\,\d_{\mathscr A}g+g\,\d_{\mathscr A}f\quad\text{ for every }f,g\in\mathscr A.
\end{equation}
\end{theorem}
\begin{proof}
This kind of construction is due to Gigli \cite{Gig:18}; see also \cite{Gig:17} or \cite{Gig:Pas:19}.
Technically speaking, the stated existence and uniqueness are direct consequences of \cite[Theorem 3.19]{Luc:Pas:23}.
The Leibniz rule \eqref{eq:Leibniz_d_A} follows from the locality \eqref{eq:locality_rel_slope} of \(|{\rm D}f|_{*,\mathscr A}\),
arguing as in \cite[Theorem 2.2.6]{Gig:18}.
\end{proof}
In analogy with \cite[Definition 2.3.1]{Gig:18}, we then define the \textbf{Sobolev \(\mathscr A\)-tangent module} as
\[
L^\infty(T\X;\mathscr A)\coloneqq L^1(T^*\X;\mathscr A)^*.
\]
Recall that \(L^\infty(T\X;\mathscr A)\) is an \(L^\infty(\mm)\)-Banach \(L^\infty(\mm)\)-module. Recall also that the operator
\[
\textsc{Int}_{\mathscr A}\coloneqq\textsc{Int}_{L^1(T^*\X;\mathscr A)}\colon L^\infty(T\X;\mathscr A)\to L^1(T^*\X;\mathscr A)'
\]
is an isometric isomorphism of Banach spaces, see \eqref{eq:Int}. To be consistent with \cite{Gig:18}, we will denote by
\(\omega(v)\in L^1(\mm)\) (instead of \(v(\omega)\in L^1(\mm)\)) the `pointwise duality pairing' between two elements \(\omega\in L^1(T^*\X;\mathscr A)\)
and \(v\in L^\infty(T\X;\mathscr A)\), even though \(L^\infty(T\X;\mathscr A)\) is the module dual of \(L^1(T^*\X;\mathscr A)\)
(and, in general, not the other way round).
\subsection{The space \texorpdfstring{\({\rm BV}_{\rm W}\)}{BVW}}
A notion of metric BV space in terms of an integration-by-parts formula involving derivations was introduced in \cite[Definition 3.1]{DiMar:14};
see also \cite{DiMaPhD:14}. We generalise it as follows:
\begin{definition}[BV space via derivations]\label{def:bvw}
Let \((\X,\sfd,\mm)\) be a metric measure space. Let \(\mathscr A\) be a unital separating subalgebra of \(\LIP_\star(\X)\).
Then we define \({\rm BV}_{\rm W}(\X;\mathscr A)\) as the space of all functions \(f\in L^1(\mm)\) for which there exists a
bounded linear operator \({\bf L}_f\colon\Der^\infty_\infty(\X;\mathscr A)\to\mathcal M(\X)\) such that
\begin{subequations}\begin{align}
\label{eq:def_BV_der_1}
{\bf L}_f(b)(\X)=-\int f\,\div(b)\,\d\mm&\quad\text{ for every }b\in{\rm Der}^\infty_\infty(\X;\mathscr A),\\
\label{eq:def_BV_der_2}
{\bf L}_f(hb)=h\,{\bf L}_f(b)&\quad\text{ for every }h\in\mathscr A\text{ and }b\in{\rm Der}^\infty_\infty(\X;\mathscr A),\\
\label{eq:def_BV_der_3}
|{\bf L}_f(b)|(\X\setminus{\rm spt}(\mm))=0&\quad\text{ for every }b\in\Der^\infty_\infty(\X;\mathscr A).
\end{align}\end{subequations}
\end{definition}
\begin{remark}\label{rmk:L_f_unique}{\rm
Given any function \(f\in{\rm BV}_{\rm W}(\X;\mathscr A)\), the operator \({\bf L}_f\) is uniquely determined. Indeed, if another
operator \(\tilde{\bf L}_f\) satisfies the same properties, then \eqref{eq:def_BV_der_2} and \eqref{eq:def_BV_der_1} imply that
\[
\int h\,\d({\bf L}_f(b)-\tilde{\bf L}_f(b))={\bf L}_f(hb)(\X)-\tilde{\bf L}_f(hb)(\X)=-\int f\,\div(hb)\,\d\mm+\int f\,\div(hb)\,\d\mm=0
\]
for every \(h\in\mathscr A\) and \(b\in{\rm Der}^\infty_\infty(\X;\mathscr A)\), whence it follows that \({\bf L}_f=\tilde{\bf L}_f\)
thanks to \eqref{eq:Stone-Weiestrass}.
\fr}\end{remark}

By Remark \ref{rmk:L_f_unique}, it makes sense to define the \textbf{total \(\mathscr A\)-variation} of any \(f\in{\rm BV}_{\rm W}(\X;\mathscr A)\) as
\[
{\rm V}_{\mathscr A}(f)\coloneqq\|{\bf L}_f\|_{\mathcal L(\overline\Der^\infty_\infty(\X;\mathscr A);\mathcal M(\X))},
\]
where we keep the notation \({\bf L}_f\) to indicate the unique bounded linear operator from \(\overline\Der^\infty_\infty(\X;\mathscr A)\)
to \(\mathcal M(\X)\) that extends \({\bf L}_f\).
One can readily check that \({\rm BV}_{\rm W}(\X;\mathscr A)\) is a Banach space if endowed with the norm
\[
\|f\|_{{\rm BV}_{\rm W}(\X;\mathscr A)}\coloneqq\|f\|_{L^1(\mm)}+{\rm V}_{\mathscr A}(f)\quad\text{ for every }f\in{\rm BV}_{\rm W}(\X;\mathscr A).
\]
\begin{proposition}\label{prop:equiv_|Df|}
Let \((\X,\sfd,\mm)\) be a metric measure space. Let \(\mathscr A\) be a unital separating subalgebra of \(\LIP_\star(\X)\).
For any \(f\in{\rm BV}_{\rm W}(\X;\mathscr A)\), we define its \textbf{total \(\mathscr A\)-variation measure} as
\[
|{\bf D}f|_{\mathscr A}\coloneqq\bigvee\big\{|{\bf L}_f(b)|\;\big|\;b\in{\rm Der}^\infty_\infty(\X;\mathscr A),\,|b|_{\mathscr A}\leq 1\big\}.
\]
Then \(|{\bf D}f|_{\mathscr A}\) is a finite non-negative Borel measure on \(\X\) such that
\(|{\bf D}f|_{\mathscr A}(\X)={\rm V}_{\mathscr A}(f)\).
\end{proposition}
\begin{remark}\label{rem:reduction}{\rm
Since \({\bf L}_f(0)\) is the null measure and \({\bf L}_f(-b)=-{\bf L}_f(b)\) for every \(b\in{\rm Der}^\infty_\infty(\X;\mathscr A)\),
the definition of $|{\bf D}f|_{\mathscr A}$ can be equivalently rewritten as
\[
|{\bf D}f|_{\mathscr A}(E) = \sup \sum_{i=1}^n {\bf L}_f(b_i)(E_i)\quad\text{ for every Borel set }E\subseteq\X,
\]
where the supremum is taken among all finite Borel partitions $(E_i)_{i=1}^n$ of $E$ and all collections $(b_i)_{i=1}^n \subseteq {\rm Der}^\infty_\infty(\X;\mathscr A)$
such that $|b_i|_{\mathscr A}\leq 1$ and $E_i$ is contained in the support of the positive part of ${\bf L}_f(b_i)$. This is because, given $b \in {\rm Der}^\infty_\infty(\X;\mathscr A)$
and a Borel subset $F \subseteq\X$, we have that
\[
|{\bf L}_f(b)|(F) = {\bf L}_f(b)(F \cap P_{b}) + {\bf L}_f(-b)(F \setminus P_{b}),
\]
where $P_{b}$ is the support of the positive part of ${\bf L}_f(b)$. Thus one can always replace the partition $(E_i)_{i=1}^n$ with
$(E_i \cap P_{b_i})_{i=1}^n \cup (E_i \setminus P_{b_i})_{i=1}^n$, and the collection $(b_i)_{i=1}^n$ with $(b_i)_{i=1}^n \cup (-b_i)_{i=1}^n$.
\fr}\end{remark}
\begin{proof}[Proof of Proposition \ref{prop:equiv_|Df|}]
Let \(\lambda<|{\bf D}f|_{\mathscr A}(\X)\) be fixed. Taking into account Remark \ref{rem:reduction},
we can find a Borel partition \(E_1,\ldots,E_n\) of \(\X\) and \(b_1,\ldots,b_n\in{\rm Der}^\infty_\infty(\X;\mathscr A)\) such that
\(|b_i|_{\mathscr A}\leq 1\) for all \(i=1,\ldots,n\) and \(\sum_{i=1}^n{\bf L}_f(b_i)(E_i)>\lambda\). Letting \(\mu\coloneqq\sum_{i=1}^n|{\bf L}_f(b_i)|\) and
\(\delta\coloneqq\sum_{i=1}^n{\bf L}_f(b_i)(E_i)-\lambda\), we know from Lemma \ref{lem:part_of_unity} that there exist non-negative functions
\(\eta_1,\ldots,\eta_n\in\mathscr A\) with \(\sum_{i=1}^n\eta_i\leq 1\) such that \(\|\eta_i-\1_{E_i}\|_{L^1(\mu)}\leq\delta/n\) for every \(i=1,\ldots,n\).
Letting \(b\coloneqq\sum_{i=1}^n\eta_i b_i\in{\rm Der}^\infty_\infty(\X;\mathscr A)\), we have that \(|b|_{\mathscr A}\leq\sum_{i=1}^n\eta_i|b_i|_{\mathscr A}\leq 1\) and
\[
\bigg|\sum_{i=1}^n{\bf L}_f(b_i)(E_i)-{\bf L}_f(b)(\X)\bigg|\leq\sum_{i=1}^n\int|\1_{E_i}-\eta_i|\,\d|{\bf L}_f(b_i)|\leq\delta,
\]
whence it follows that
\[
\lambda=\sum_{i=1}^n{\bf L}_f(b_i)(E_i)-\delta\leq{\bf L}_f(b)(\X)\leq|{\bf L}_f(b)|(\X)\leq{\rm V}_{\mathscr A}(f).
\]
Letting \(\lambda\nearrow|{\bf D}f|_{\mathscr A}(\X)\), we deduce that \(|{\bf D}f|_{\mathscr A}(\X)\leq{\rm V}_{\mathscr A}(f)\), thus
\(|{\bf D}f|_{\mathscr A}\) is finite. Finally, we have \(|{\bf L}_f(b)|(\X)\leq|{\bf D}f|_{\mathscr A}(\X)\) for all
\(b\in{\rm Der}^\infty_\infty(\X;\mathscr A)\) with \(|b|_{\mathscr A}\leq 1\), thus \({\rm V}_{\mathscr A}(f)\leq|{\bf D}f|_{\mathscr A}(\X)\).
\end{proof}
\begin{remark}\label{rmk:|Df|_Radon}{\rm
Assuming in addition that \((\X,\sfd)\) is complete, it holds that
\[
|{\bf D}f|_{\mathscr A}\text{ is a Radon measure for every }f\in{\rm BV}_{\rm W}(\X;\mathscr A).
\]
Indeed, \(|{\bf D}f|_{\mathscr A}\) is concentrated on \({\rm spt}(\mm)\) as a consequence of \eqref{eq:def_BV_der_3},
thus it is a Radon measure because \({\rm spt}(\mm)\) is complete and separable.
\fr}\end{remark}

When \(\mathscr A=\LIP_\star(\X)\), we just write \({\rm BV}_{\rm W}(\X)\) and \(|{\bf D}f|\) for brevity.
\begin{lemma}\label{lem:BV_H_loc_in_BV_W}
Let \((\X,\sfd,\mm)\) be a metric measure space. Assume that \((\X,\sfd)\) is either complete or a Radon space. Then it holds that
\({\rm BV}_{\rm H}(\X)\subseteq{\rm BV}_{\rm W}(\X)\) and
\[
|{\bf D}f|\leq|{\bf D}f|_*\quad\text{ for every }f\in{\rm BV}_{\rm H}(\X).
\]
\end{lemma}
\begin{proof}
Let \(f\in{\rm BV}_{\rm H}(\X)\) be given. Fix any \(b\in{\rm Der}^\infty_\infty(\X)\). Take a sequence \((f_n)_n\subseteq\LIP_\star(\X)\) such that
\(f_n\to f\) in \(L^1(\mm)\) and \(\int\lip_a(f_n)\,\d\mm\to\|{\bf D}f\|_*\). Accounting for Remark \ref{rmk:Portmanteau}, we have that
\(\lip_a(f_n)\mm\rightharpoonup|{\bf D}f|_*\) weakly. Since \(|{\bf D}f|_*\) is a Radon measure by Remark \ref{rmk:|Df|_*_Radon}, we deduce from
\cite[Theorem 8.6.4]{Bog:07} that \((\lip_a(f_n)\mm)_n\) is a tight sequence.
As \(|b(f_n)\mm|=|b(f_n)|\mm\leq|b|\lip_a(f_n)\mm\) for every \(n\in\N\), we deduce that \((b(f_n)\mm)_n\) is a tight sequence as well.
Since \((b(f_n)\mm)_n\) is also bounded in total variation norm, by \cite[Theorem 8.6.7]{Bog:07} we deduce that, up to a non-relabelled
subsequence, \(b(f_n)\mm\rightharpoonup{\bf L}_f(b)\) weakly for some \({\bf L}_f(b)\in\mathcal M(\X)\). For any \(h\in\LIP_\star(\X)\),
we can compute
\[
\int h\,\d{\bf L}_f(b)=\lim_{n\to\infty}\int h\,b(f_n)\,\d\mm=-\lim_{n\to\infty}\int f_n(h\,\div(b)+b(h))\,\d\mm=-\int f(h\,\div(b)+b(h))\,\d\mm.
\]
Therefore, the measure \({\bf L}_f(b)\) is independent of the approximating sequence and \(b(f_n)\mm\rightharpoonup{\bf L}_f(b)\) holds for the
original sequence \((f_n)_n\). In particular, we have that \({\bf L}_f(b)\) is concentrated on \({\rm spt}(\mm)\),
proving \eqref{eq:def_BV_der_3}. If \(b,\tilde b\in\Der^\infty_\infty(\X)\) and \(\lambda\in\R\),
then \((\lambda b+\tilde b)(f_n)\mm=\lambda\,b(f_n)\mm+\tilde b(f_n)\mm\) weakly converges both to
\({\bf L}_f(\lambda b+\tilde b)\) and \(\lambda\,{\bf L}_f(b)+{\bf L}_f(\tilde b)\), so that \({\bf L}_f\) is linear.
Given \(h,\tilde h\in\LIP_\star(\X)\), we have that \(\int\tilde h\,\d{\bf L}_f(hb)=\lim_n\int\tilde h(hb)(f_n)\,\d\mm
=\lim_n\int(\tilde h h)b(f_n)\,\d\mm=\int\tilde h h\,\d{\bf L}_f(b)=\int\tilde h\,\d(h\,{\bf L}_f(b))\), which yields
\eqref{eq:def_BV_der_2} thanks to \eqref{eq:Stone-Weiestrass}. Moreover, for every \(b\in\Der^\infty_\infty(\X)\)
we can compute
\[
{\bf L}_f(b)(\X)=\lim_{n\to\infty}\int b(f_n)\,\d\mm=-\lim_{n\to\infty}\int f_n\,\div(b)\,\d\mm=-\int f\,\div(b)\,\d\mm,
\]
which gives \eqref{eq:def_BV_der_1}. Next, fix a derivation \(b\in\Der^\infty_\infty(\X)\) with \(|b|\leq 1\)
and a compact set \(K\subseteq\X\). We claim that
\begin{equation}\label{eq:dimarino}
{\bf L}_f(b)(K)\leq|{\bf D}f|_*(K),    
\end{equation}
whence (since \(|{\bf L}_f(b)|\) is a Radon measure) it follows that \(f\in{\rm BV}_{\rm W}(\X)\) and
\(|{\bf D}f|\leq|{\bf D}f|_*\). To prove the claim, pick a sequence of open sets $(U_i)_i$ such that $K \subset U_i$ for every $i \in \N$ and satisfying ${\bf L}_f(b)(U_i)\to{\bf L}_f(b)(K)$, $|{\bf D}f|_*(U_i) \to |{\bf D}f|_*(K)$.
Now fix \(i\in\N\) and take a non-decreasing sequence \((\eta_{i,j})_j\) of Lipschitz cut-off functions
\(\eta_{i,j}\colon\X\to[0,1]\) such that \(\eta_{i,j}(x)\nearrow\1_{U_i}(x)\) as \(j\to\infty\)
for every \(x\in\X\). Note that 
\[\begin{split}
\int\eta_{i,j}\,\d{\bf L}_f(b)&=\lim_{n\to\infty}\int\eta_{i,j}\,b(f_n)\,\d\mm
\leq\liminf_{n\to\infty}\int\eta_{i,j}|b|\,\lip_a(f_n)\,\d\mm\\
& \leq\liminf_{n\to\infty}\int\eta_{i,j}\,\lip_a(f_n)\,\d\mm  = \int\eta_{i,j}\,\d |{\bf D}f|_* \leq |{\bf D}f|_*(U_i)
\end{split}\]
for every \(i,j\in\N\). Applying the dominated convergence theorem, we deduce that
\[
{\bf L}_f(b)(U_i)=\lim_{j\to\infty}\int\eta_{i,j}\,\d{\bf L}_f(b)\leq|{\bf D}f|_*(U_i)\quad\text{ for every }i\in\N.
\]
Therefore, we can finally conclude that
\[
{\bf L}_f(b)(K)=\lim_{i\to\infty}{\bf L}_f(b)(U_i)\leq\lim_{i\to\infty}|{\bf D}f|_*(U_i)=|{\bf D}f|_*(K),
\]
which gives \eqref{eq:dimarino}. Thus, the proof is complete.
\end{proof}
\begin{remark}\label{rmk:BV_H_loc_in_BV_W_bis}{\rm
In the case where \((\X,\sfd)\) is complete, by slightly adapting the proof of Lemma \ref{lem:BV_H_loc_in_BV_W}
(and taking also Proposition \ref{prop:approx_with_glob_Lip} into account) one can prove that
\({\rm BV}_{\rm H}(\X)\subseteq{\rm BV}_{\rm W}(\X;\LIP_b(\X))\) and \(|{\bf D}f|_{\LIP_b(\X)}\leq|{\bf D}f|_*\)
for every \(f\in{\rm BV}_{\rm H}(\X)\).
\fr}\end{remark}
\section{Equivalence results}
Given a metric space \((\X,\sfd)\) and \(x\in\X\), we define the \(1\)-Lipschitz function \(\sfd_x\colon\X\to[0,+\infty)\) as
\[
\sfd_x(y)\coloneqq\sfd(x,y)\quad\text{ for every }y\in\X.
\]
\begin{definition}[Good algebra]\label{def:good_algebra}
Let \((\X,\sfd,\mm)\) be a metric measure space. Let \(\mathscr A\) be a unital separating subalgebra
of \(\LIP_b(\X)\).
Then we say that \(\mathscr A\) is a \textbf{good algebra} if, given any \(x\in\X\) and \(r>0\), there exist
\((f_n)_n\subseteq\mathscr A\) and \(\rho\in L^\infty(\mm)\) with \(0\leq\rho\leq 1\) such that
\(\|f_n-\sfd_x\wedge r\|_{L^1(\mm)}\to 0\) and \(\lip_a(f_n)\rightharpoonup\rho\) weakly in \(L^1(\mm)\).
\end{definition}

Note that a unital separating subalgebra \(\mathscr A\) of \(\LIP_b(\X)\) is a good algebra precisely when
\[
\sfd_x\wedge r\in{\rm H}^{1,1}(\X;\mathscr A),\qquad|{\rm D}(\sfd_x\wedge r)|_{*,\mathscr A}\leq 1
\quad\text{ for every }x\in \X\text{ and }r>0.
\]
Clearly, \(\LIP_b(\X)\) itself is a good algebra. 
\begin{example}[Some good algebras]\label{ex:good_alg}
{\rm In \cite{Sav:22} is introduced a notion of \emph{compatible} algebra in the setting of
\emph{extended metric-topological measure spaces}. These are Hausdorff topological spaces $(\X, \tau)$
endowed with a (possibly extended) metric $\sfd$ which has good compatibility properties with $\tau$,
together with a finite, non-negative, Radon measure $\mm$ (on $(\X, \tau)$). Whenever $\tau$ coincides
with the topology induced by $\sfd$ and \(\sfd\) is finite, we recover our setting, i.e.\ $(\X,\sfd,\mm)$
is a metric measure space in the sense of Definition \ref{def:mms}. In this case, an algebra $\mathscr A$
is said to be compatible if
\[
\sfd(x,y) = \sup\left \{ \frac{|f(x)-f(y)|}{\Lip(f)} \;\bigg|\;f \in \mathscr A,\,\Lip(f)\neq 0\right\}
\quad\text{ for every }x,y \in \X.
\]
It is not difficult to check that this implies that the functions $\sfd_x$, $x \in\X$ can be approximated in the following sense: there exists a sequence $(f_n)_n \subset \mathscr A$ such that
\[
f_n \to \sfd_x \quad \text{$\mm$-a.e.}, \quad \lip_a(f_n) \rightharpoonup \rho \text{ in 
} L^1(\mm) \quad \text{ for some } 0 \le \rho \le 1.
\]
This is precisely (even if there is considered only for $p>1$) the density condition in \cite[Theorem 2.13]{Fo:Sa:So:23}.
In turn this last condition implies that $\mathscr A$ is a good algebra according to Definition \ref{def:good_algebra},
simply using a truncation argument of the distance. Therefore, a variety of algebras in many metric spaces are good algebras:
\begin{enumerate}
\item Smooth functions in Riemannian manifolds and Euclidean spaces \cite[Remark 2.21]{Fo:Sa:So:23}.
\item Cylinder functions in Banach and Hilbert spaces \cite[Example 2.1.19]{Sav:22}.
\item Algebras generated by approximating distances \cite[Example 2.1.20]{Sav:22}.
\item Cylinder functions in the Wasserstein, Hellinger and Hellinger--Kantorovich metric spaces of (probability) measures \cite{Sod:23, Fo:Sa:So:23, Ds:So:25}. \fr
\end{enumerate}
}\end{example}
\begin{remark}\label{rmk:improved_good_alg}{\rm
If \(\mathscr A\) is a good algebra, then the approximation \((f_n)_n\) of \(\sfd_x\wedge r\) can be also chosen so that
\(0\leq f_n\leq r\) for every \(n\in\N\) and $\lip_a(f_n) \to |{\rm D}(\sfd_x\wedge r)|_{*,\mathscr A}$ strongly in $L^1(\mm)$.
To prove it, we first take a sequence \((g_n)_n\subseteq\mathscr A\) such that \(g_n\to\sfd_x\wedge r\)
in \(L^1(\mm)\) and \(\lip_a(g_n)\leq G_n\) for some \((G_n)_n\) strongly converging in \(L^1(\mm)\) to
\(|{\rm D}(\sfd_x\wedge r)|_{*,\mathscr A}\), whose existence stems from Mazur's lemma and the convexity
of \(f\mapsto\lip_a(f)\). Note that \(u_n\coloneqq(g_n\wedge r)\vee 0\to\sfd_x\wedge r\) in \(L^1(\mm)\). Thanks to
\cite[Lemma 2.1.26]{Sav:22}, we can find a sequence \((\tilde f_n)_n\subseteq\mathscr A\) such that
\(|\tilde f_n-u_n|\leq\frac{1}{n}\) and \(\lip_a(\tilde f_n)\leq\lip_a(u_n)\leq\lip_a(g_n)\leq G_n\) for all \(n\in\N\).
By the Dunford--Pettis theorem and the minimality of $|{\rm D}(\sfd_x\wedge r)|_{*,\mathscr A}$, we thus have
(up to a non-relabelled subsequence) that \(\lip_a(\tilde f_n)\rightharpoonup|{\rm D}(\sfd_x\wedge r)|_{*,\mathscr A}\)
weakly in \(L^1(\mm)\). In particular, testing against the constant function \(\1_\X\in L^\infty(\mm)\) we deduce that
\(\int_\X\lip_a(\tilde{f}_n)\,\d\mm\to\int_\X|{\rm D}(\sfd_x\wedge r)|_{*,\mathscr A} \, \d \mm\), so that
\[
c_n\coloneqq\frac{\int_\X|{\rm D}(\sfd_x\wedge r)|_{*,\mathscr A} \, \d \mm}{\int_\X\lip_a(\tilde{f}_n)\,\d\mm}\to 1
\quad\text{ as }n\to\infty.
\]
Therefore, the functions $\bar{f}_n \coloneqq c_n\tilde f_n\in\mathscr A$ satisfy $\bar{f}_n \to \sfd_x \wedge r$ strongly
in $L^1(\mm)$ as \(n\to\infty\), along with $\int_\X\lip_a(\bar f_n)\,\d\mm=\int_\X|{\rm D}(\sfd_x\wedge r)|_{*,\mathscr A}\,\d\mm$ for every \(n\in\N\).
In particular, we have that
\begin{align*}
    \int_\X | \lip_a(\bar f_n) - |{\rm D}(\sfd_x\wedge r)|_{*,\mathscr A}| \, \d \mm &= 2 \int_\X ( \lip_a(\bar f_n) - |{\rm D}(\sfd_x\wedge r)|_{*,\mathscr A})^+ \, \d \mm \\
    & \le \int_\X ( c_nG_n - |{\rm D}(\sfd_x\wedge r)|_{*,\mathscr A})^+ \, \d \mm \to 0.
\end{align*}
Finally, letting
\(f_n\coloneqq c_n^{-1}r\big(r+\frac{2}{n}\big)^{-1}\big(\bar f_n+\frac{c_n}{n}\big)\) for all \(n\in\N\),
we obtain the desired sequence.}
\fr\end{remark}
\begin{proposition}\label{prop:approx_with_A}
Let \((\X,\sfd,\mm)\) be a metric measure space and \(\mathscr A\) a good subalgebra of \(\LIP_b(\X)\).
Let \(f\in\LIP_b(\X)\) be given. Then there exists a sequence \((f_n)_n\subseteq\mathscr A\) such that the following hold:
\begin{itemize}
\item[\(\rm i)\)] \(\inf_X f\leq f_n\leq\sup_\X f\) for every \(n\in\N\).
\item[\(\rm ii)\)] \(f_n\to f\) in \(L^1(\mm)\) as \(n\to\infty\).
\item[\(\rm iii)\)] \((\lip_a(f_n))_n\) is \textbf{dominated} in \(L^1(\mm)\), i.e.\ there exists \(G\in L^1(\mm)^+\)
such that \(\lip_a(f_n)\leq G\) holds \(\mm\)-a.e.\ on \(\X\) for every \(n\in\N\).
\item[\(\rm iv)\)] \((\lip_a(f_n)-\Lip(f))^+\to 0\) in \(L^1(\mm)\) as \(n\to \infty\).
\end{itemize}
\end{proposition}
\begin{proof}
Without loss of generality, we can assume that \(\Lip(f)>0\). Fix a sequence \((x_k)_k\subseteq\X\) that is
dense in \({\rm spt}(\mm)\). Define \(L\coloneqq\Lip(f)\), \(m\coloneqq\inf_\X f\), \(M\coloneqq\sup_\X f\) and
\[
\tilde g_k\coloneqq(f(x_k)-L\sfd_{x_k})\vee m=f(x_k)-L\bigg(\sfd_{x_k}\wedge\frac{f(x_k)-m}{L}\bigg)
\quad\text{ for every }k\in\N.
\]
Notice that \(\tilde g_k\in\LIP_b(\X)\) and \(\Lip(\tilde g_k)\leq L\). Since \(\mathscr A\) is a good algebra, by Remark \ref{rmk:improved_good_alg}, 
we can find \((f^j_k)_j\subseteq\mathscr A\) such that \(f^j_k\to\tilde g_k\), 
\(\lip_a(f^j_k)\to |{\rm D}\tilde g_k|_{*,\mathscr A} \leq L\) in \(L^1(\mm)\) as \(j\to\infty\), and \(m\leq f^j_k\leq M\) for every \(k,j\in\N\). Define
\[
g_n\coloneqq\tilde g_1\vee\ldots\vee\tilde g_n\in\LIP_b(\X)\quad\text{ for every }n\in\N.
\]
Given that \(m\leq g_n\leq M\) for every \(n\in\N\) and \(g_n(x)\to f(x)\) for every \(x\in{\rm spt}(\mm)\),
we have that \(g_n\to f\) in \(L^1(\mm)\) by the dominated convergence theorem. Now fix \(n\in\N\). Since
\(f^j_1\vee\ldots\vee f^j_n\to g_n\) in \(L^1(\mm)\) as \(j\to\infty\), there exists \(j(n)\in\N\) such that
\(\|f^{j(n)}_1\vee\ldots\vee f^{j(n)}_n-g_n\|_{L^1(\mm)}\leq 1/n\) and
\[
\| h_n^k\|_{L^1(\mm)} \le \frac{1}{n2^n}\quad\text{ for every }k=1,\ldots,n,
\text{ where } h_n^k\coloneqq\big| \lip_a(f_k^{j(n)})-|{\rm D}\tilde g_k|_{*,\mathscr A}\big|.
\]
By \cite[Lemma 2.1.26]{Sav:22}, we can find a function \(\tilde f_n\in\mathscr A\) such that
\(\sup_\X|f^{j(n)}_1\vee\ldots\vee f^{j(n)}_n-\tilde f_n|\leq 1/n\) and
\(\lip_a(\tilde f_n)\leq\lip_a(f_1^{j(n)})\vee\ldots\vee\lip_a(f_n^{j(n)})\). In particular,
\(m-\frac{1}{n}\leq\tilde f_n\leq M+\frac{1}{n}\) for every \(n\in\N\). It also follows that
\(\|\tilde f_n-g_n\|_{L^1(\mm)}\leq(1+\mm(\X))/n\) and
\[
\lip_a(\tilde f_n) \le ( |{\rm D}\tilde g_1|_{*,\mathscr A} +h_n^1) \vee \dots \vee
( |{\rm D}\tilde g_n|_{*,\mathscr A} +h_n^n) \le L + h_n^1 \vee \dots \vee h_n^n \leq L + h_n^1 + \dots + h_n^n.
\]
Note that $h_n\coloneqq h_n^1 + \dots + h_n^n$ satisfies $\|h_n\|_{L^1(\mm)} \le 2^{-n}$, thus $h\coloneqq\sum_{n=1}^\infty h_n$
belongs to $L^1(\mm)$. Therefore, $\tilde f_n \to f$ in $L^1(\mm)$ as $n \to \infty$ and $\lip_a(\tilde f_n)\le L+h \in L^1(\mm)$
for every $n \in \N$. Moreover, we have that $(\lip_a(\tilde f_n) - L)^+ \le h_n \to 0$ in $L^1(\mm)$ as $n \to \infty$.
Finally, defining \((f_n)_n\subseteq\mathscr A\) as
\[
f_n\coloneqq c_n\bigg(\tilde f_n-m+\frac{1}{n}\bigg)+m,\quad\text{ where we set }c_n\coloneqq\frac{M-m}{M-m+\frac{2}{n}},
\]
one can then easily check that \((f_n)_n\) satisfies i), ii), iii) and iv). The statement is achieved.
\end{proof}
The following two results contain the first steps to prove the equivalence of definitions of BV spaces in Theorem \ref{thm:equivalence_A}, in particular the inclusion \({\rm BV}_{\rm W}(\X;\LIP_b(\X)) \subseteq {\rm BV}_{\rm W}(\X;\mathscr A)\) contained in Corollary \ref{cor:BV_W_indep_A}, which is an immediate consequence of the next Theorem \ref{thm:indep_der}.
\begin{theorem}\label{thm:indep_der}
Let \((\X,\sfd,\mm)\) be a metric measure space. Let \(\mathscr A\) be a good subalgebra of $\LIP_b(\X)$. Define the operator
\(\phi_{\mathscr A}\colon{\rm Der}^\infty_\infty(\X;\LIP_b(\X))\to{\rm Der}^\infty_\infty(\X;\mathscr A)\) as
\[
\phi_{\mathscr A}(b)\coloneqq b|_{\mathscr A}\in{\rm Der}^\infty_\infty(\X;\mathscr A)
\quad\text{ for every }b\in{\rm Der}^\infty_\infty(\X;\LIP_b(\X)).
\]
Then \(\phi_{\mathscr A}\) is an isomorphism of \(\mathscr A\)-modules such that the identity \(\div(\phi_{\mathscr A}(b))=\div(b)\)
holds for every \(b\in\Der^\infty_\infty(\X;\LIP_b(\X))\). Moreover, it holds that
\[
|\phi_{\mathscr A}(b)|_{\mathscr A}=|b|\quad\text{ for every }b\in{\rm Der}^\infty_\infty(\X;\LIP_b(\X)).
\]
\end{theorem}
\begin{proof}
Let us write \({\rm L}_\X\coloneqq\LIP_b(\X)\)
for brevity. First of all, notice that \(\phi_{\mathscr A}(b)\in\Der^\infty_\infty(\X;\mathscr A)\) and
\(|\phi_{\mathscr A}(b)|_{\mathscr A}\leq|b|\) for every \(b\in\Der^\infty_\infty(\X;{\rm L}_\X)\).
Moreover, the operator \(\phi_{\mathscr A}\) is a homomorphism of \(\mathscr A\)-modules by construction.

To conclude, it remains to show that if \(b\in\Der^\infty_\infty(\X;\mathscr A)\) is given, there exists
\(\bar b\in\Der^\infty_\infty(\X;{\rm L}_\X)\) such that \(\phi_{\mathscr A}(\bar b)=b\) and \(|\bar b|\leq|b|_{\mathscr A}\).
To prove it, fix a function \(f\in{\rm L}_\X\) and take a sequence \((f_n)_n\subseteq\mathscr A\) that approximates
\(f\) in the sense of Proposition \ref{prop:approx_with_A}. Since \(|b(f_n)|\leq|b|_{\mathscr A}\lip_a(f_n)\)
for every \(n\in\N\), we have that \((b(f_n))_n\) is dominated in $L^1(\mm)$, thus the Dunford--Pettis theorem ensures that
we can extract a subsequence \((f_{n_k})_k\) such that \(b(f_{n_k})\rightharpoonup\bar b(f)\) in $L^1(\mm)$, for some limit
function \(\bar b(f)\in L^1(\mm)\). Note that whenever a sequence \((\tilde f_i)_i\subseteq\mathscr A\) and a function
$h \in L^1(\mm)$ satisfy \(\tilde f_i\rightharpoonup f\) and \(b(\tilde f_i)\rightharpoonup h\)
weakly in \(L^1(\mm)\), for any \(g\in\mathscr A\) we have that
\begin{equation}\label{eq:ext_der_aux1}\begin{split}
\int hg\,\d\mm&=\lim_{i\to\infty}\int g\,b(\tilde f_i)\,\d\mm=\lim_{i\to\infty}\int b(g\tilde f_i)-\tilde f_i\,b(g)\,\d\mm\\
&=-\lim_{i\to\infty}\int\tilde f_i(g\,\div(b)+b(g))\,\d\mm=-\int f(g\,\div(b)+b(g))\,\d\mm.
\end{split}\end{equation}
This shows that \(\bar b(f)\) is independent of the chosen subsequence \((f_{n_k})_k\) and that the following holds:
if \((\tilde f_i)_i\subseteq\mathscr A\) satisfies \(\tilde f_i\rightharpoonup f\) weakly in \(L^1(\mm)\) and
\((b(\tilde f_i))_i\) is dominated in $L^1(\mm)$, then \(b(\tilde f_i)\rightharpoonup\bar b(f)\) weakly in $L^1(\mm)$.
In particular, we have that \(\bar b|_{\mathscr A}=b\). Let us now check that \(\bar b\in{\rm Der}^\infty_\infty(\X;{\rm L}_\X)\):
\begin{itemize}
\item Fix \(f,g\in{\rm L}_\X\) and \(\alpha,\beta\in\R\). Let \((f_n)_n\subseteq\mathscr A\)
(resp.\ \((g_n)_n\subseteq\mathscr A\)) be an approximating sequence for \(f\) (resp.\ \(g\))
as in Proposition \ref{prop:approx_with_A}. Letting \(n\to\infty\) in the identity
\(b(\alpha f_n+\beta g_n)=\alpha\,b(f_n)+\beta\,b(g_n)\), we obtain that
\(\bar b(\alpha f+\beta g)=\alpha\,\bar b(f)+\beta\,\bar b(g)\), proving that \(\bar b\) is linear.
\item Note that whenever $f \in {\rm L}_\X$ and $(f_n)_n \subseteq \mathscr A$ is an approximating sequence for $f$ as in Proposition \ref{prop:approx_with_A}, we have that
\begin{equation}\label{eq:ext_der_aux2}
\int\bar b(f)\,\d\mm=\lim_{n\to\infty}\int b(f_n)\,\d\mm=-\lim_{n\to\infty}\int f_n\,\div(b)\,\d\mm=-\int f\,\div(b)\,\d\mm.
\end{equation}
\item We claim that
\begin{equation}\label{eq:ext_der_aux3}
\bar b(fg)=f\,b(g)+g\,\bar b(f)\quad\text{ for every }f\in{\rm L}_\X\text{ and }g\in\mathscr A.
\end{equation}
Indeed, letting \((f_n)_n\subseteq\mathscr A\) be an approximating sequence for \(f\) as in Proposition \ref{prop:approx_with_A},
we have that \(b(f_n g)=f_n \,b(g)+g\,b(f_n)\) for every \(n\in\N\), which gives \(\bar b(fg)=f\,b(g)+g\,\bar b(f)\).
\item Next, we claim that \(\bar b\) has the following property: if \(f\in{\rm L}_\X\) and \((\tilde f_i)_i\subseteq{\rm L}_\X\)
are such that \(\tilde f_i\rightharpoonup f\) weakly in \(L^1(\mm)\) and \((\bar b(\tilde f_i))_i\) is dominated in $L^1(\mm)$,
then \(\bar b(\tilde f_i)\rightharpoonup\bar b(f)\) weakly in $L^1(\mm)$.
To prove it, one can argue as in \eqref{eq:ext_der_aux1}, by using \eqref{eq:ext_der_aux2} and \eqref{eq:ext_der_aux3}.
\item Consequently, we can show that \(\bar b\) satisfies the Leibniz rule, i.e.
\begin{equation}\label{eq:ext_der_aux4}
\bar b(fg)=f\,\bar b(g)+g\,\bar b(f)\quad\text{ for every }f,g\in{\rm L}_\X.
\end{equation}
Indeed, letting \((g_n)_n\subseteq\mathscr A\) be an approximating sequence for \(g\) as in Proposition \ref{prop:approx_with_A},
we know from \eqref{eq:ext_der_aux3} that \(\bar b(fg_n)=f\,b(g_n)+g_n\bar b(f)\) for every \(n\in\N\), whence
\eqref{eq:ext_der_aux4} follows.
\item We also claim that
\begin{equation}\label{eq:ext_der_aux5}
|\bar b(f)|\leq\Lip(f)|b|_{\mathscr A}\quad\text{ for every }f\in{\rm L}_\X.
\end{equation}
To prove it, take an approximating sequence \((f_n)_n\subseteq\mathscr A\) for \(f\) as in Proposition \ref{prop:approx_with_A}
and set $h_n\coloneqq(\lip_a(f_n)-\Lip(f))^+$. Letting \(n\to\infty\) in
\(\bar b(f_n)\leq|b|_{\mathscr A}\lip_a(f_n)\leq|b|_{\mathscr A}(\Lip(f)+h_n)\), we get \(\bar b(f)\leq\Lip(f)|b|_{\mathscr A}\).
Doing the same for \(-f\) instead, we obtain that \(-\bar b(f)\leq\Lip(f)|b|_{\mathscr A}\), proving \eqref{eq:ext_der_aux5}.
\item Fix an equi-bounded, equi-Lipschitz sequence \((f_n)_n\subseteq{\rm L}_\X\) converging pointwise to \(f\in{\rm L}_\X\).
Since \(|\bar b(f_n)|\leq\Lip(f_n)|b|_{\mathscr A}\) by \eqref{eq:ext_der_aux5}, the sequence \((\bar b(f_n))_n\) is bounded
in \(L^\infty(\mm)\), thus it is dominated in $L^1(\mm)$. Given that \(f_n\to f\) in \(L^1(\mm)\) (by the dominated convergence
theorem), we have that \(\bar b(f_n)\rightharpoonup\bar b(f)\) weakly in $L^1(\mm)$ by the fourth point above.
Since \(L^1(\mm)\) is separable, by the Banach--Alaoglu theorem we can extract a subsequence \((f_{n_k})_k\) such that
\((\bar b(f_{n_k}))_k\) has a weak\(^*\) limit \(G\) in \(L^\infty(\mm)\). In particular, \(\bar b(f_{n_k})\rightharpoonup G\)
weakly in $L^1(\mm)$, thus \(G=\bar b(f)\) and accordingly the original sequence \(\bar b(f_n)\) weakly\(^*\) converges to
\(\bar b(f)\) in \(L^\infty(\mm)\). This proves that \(\bar b\) is weakly\(^*\) continuous (in the sense of
\cite[Lemma 4.12 i)]{Amb:Iko:Luc:Pas:24}). By suitably adapting the proof of \cite[Corollary 4.14]{Amb:Iko:Luc:Pas:24}
(exploiting the fact that \({\rm spt}(\mm)\) is separable), we can conclude that \(|\bar b(f)|\leq|b|_{\mathscr A}\lip_a(f)\)
for every \(f\in{\rm L}_\X\).
\end{itemize}
All in all, we showed that \(\bar b\in\Der^\infty_\infty(\X;{\rm L}_\X)\), \(\phi_{\mathscr A}(\bar b)=b\)
and \(|\bar b|\leq|b|_{\mathscr A}\). The proof is complete.
\end{proof}
\begin{corollary}\label{cor:BV_W_indep_A}
Let \((\X,\sfd,\mm)\) be a metric measure space. Let \(\mathscr A\) be a good subalgebra of \(\LIP_b(\X)\).
Then \({\rm BV}_{\rm W}(\X;\LIP_b(\X)) \subseteq {\rm BV}_{\rm W}(\X;\mathscr A)\) and
\(|{\bf D}f|_{\mathscr A} \le|{\bf D}f|_{\LIP_b(\X)}\) for all \(f\in{\rm BV}_{\rm W}(\X;\LIP_b(\X))\).
\end{corollary}
\begin{proof}
Let $f \in {\rm BV}_{\rm W}(\X;\LIP_b(\X))$ and let \({\bf L}_f\colon\Der^\infty_\infty(\X;\LIP_b(\X))\to\mathcal M(\X)\)
be as in Definition \ref{def:bvw} for $f$ and the algebra $\LIP_b(\X)$. We define the operator
\(\tilde{\bf L}_f\colon\Der^\infty_\infty(\X; \mathscr A) \to\mathcal M(\X)\) as
\[
\tilde{\bf L}_f(b):= {\bf L}_f(\phi_{\mathscr A}^{-1}(b))\quad \text{ for every }b \in \Der^\infty_\infty(\X; \mathscr A),
\]
where $\phi_{\mathscr A}$ is as in Theorem \ref{thm:indep_der}. It is easy to check that $\tilde{\bf L}_f$ satisfies
Definition \ref{def:bvw} for $f$ and the algebra $\mathscr A$. The inequality between the total variation measures
follows immediately from the properties of $\phi_{\mathscr A}$.
\end{proof}

Let us now fix a unital separating subalgebra \(\mathscr A\) of \(\LIP_\star(\X)\).
We can regard \(\d_{\mathscr A}\) as an unbounded linear operator \(\d_{\mathscr A}\colon L^1(\mm)\to L^1(T^*\X;\mathscr A)\)
whose domain is \(D(\d_{\mathscr A})=\mathscr A\). Given that \(\mathscr A\) is dense in \(L^1(\mm)\), we have that
\(\d_{\mathscr A}\) is densely defined, thus accordingly its adjoint operator
\[
\d_{\mathscr A}^*\colon L^1(T^*\X;\mathscr A)'\to L^\infty(\mm)
\]
is well defined. Recall that the domain \(D(\d_{\mathscr A}^*)\) of \(\d_{\mathscr A}^*\) is a vector subspace
of \(L^1(T^*\X;\mathscr A)'\), and that \(\d_{\mathscr A}^*\) is characterised by \(\int f\,\d_{\mathscr A}^*V\,\d\mm=V(\d_{\mathscr A}f)\)
for every \(V\in D(\d_{\mathscr A}^*)\) and \(f\in\mathscr A\).
\begin{lemma}\label{lem:D(d*)_to_Der}
Let \((\X,\sfd,\mm)\) be a metric measure space. Let \(\mathscr A\) be a unital separating subalgebra of \(\LIP_\star(\X)\).
Fix any \(v\in L^\infty(T\X;\mathscr A)\) with \(\textsc{Int}_{\mathscr A}(v)\in D(\d_{\mathscr A}^*)\).
Let us define \(b_v\colon\mathscr A\to L^1(\mm)\) as 
\[
b_v(f)\coloneqq\d_{\mathscr A}f(v)\in L^1(\mm)\quad\text{ for every }f\in\mathscr A.
\]
Then \(b_v\in\Der^\infty_\infty(\X;\mathscr A)\). Moreover, it holds that \(|b_v|_{\mathscr A}\leq|v|\) and \(\div(b_v)=-\d_{\mathscr A}^*(\textsc{Int}_{\mathscr A}(v))\).
\end{lemma}
\begin{proof}
First, observe that \(\d_{\mathscr A}f(v)\leq|v||\d_{\mathscr A}f|\leq|v|\lip_a(f)\) for every \(f\in\mathscr A\), so that
\(b_v(f)\in L^1(\mm)\) and \(b_v\) satisfies Definition \ref{def:der} i) with \(G=|v|\). The linearity of \(b_v\) follows from the
linearity of \(\d_{\mathscr A}\). For any \(f,g\in\mathscr A\), we deduce from \eqref{eq:Leibniz_d_A} that
\[
b_v(fg)=\d_{\mathscr A}(fg)(v)=f\,\d_{\mathscr A}g(v)+g\,\d_{\mathscr A}f(v)=f\,b_v(g)+g\,b_v(f),
\]
which proves that \(b_v\) satisfies Definition \ref{def:der} ii), thus accordingly \(b_v\in\Der^\infty(\X;\mathscr A)\) and \(|b_v|_{\mathscr A}\leq|v|\).
Finally, recalling the definitions of \(\d_{\mathscr A}^*\) and \(\textsc{Int}_{\mathscr A}\), for any \(f\in\mathscr A\) we can compute
\[
\int b_v(f)\,\d\mm=\int\d_{\mathscr A}f(v)\,\d\mm=\textsc{Int}_{\mathscr A}(v)(\d_{\mathscr A}f)=\int f\,\d_{\mathscr A}^*(\textsc{Int}_{\mathscr A}(v))\,\d\mm,
\]
which gives that \(b_v\in\Der^\infty_\infty(\X;\mathscr A)\) and \(\div(b_v)=-\d_{\mathscr A}^*(\textsc{Int}_{\mathscr A}(v))\). The proof is thus complete.
\end{proof}

We now define the \textbf{pre-Cheeger \(\mathscr A\)-energy} functional \(\tilde{\mathcal E}_{\mathscr A}\colon L^1(\mm)\to[0,+\infty]\) as
\[
\tilde{\mathcal E}_{\mathscr A}(f)\coloneqq\left\{\begin{array}{ll}
\int|{\rm D}f|_{*,\mathscr A}\,\d\mm\\
+\infty
\end{array}\quad\begin{array}{ll}
\text{ if }f\in\mathscr A,\\
\text{ otherwise.}
\end{array}\right.
\]
It can happen that \(\tilde{\mathcal E}_{\mathscr A}|_{\mathscr A}\neq\mathcal E_{\mathscr A}|_{\mathscr A}\), but \(\mathcal E_{\mathscr A}\) is the
\(L^1(\mm)\)-lower semicontinuous envelope of \(\tilde{\mathcal E}_{\mathscr A}\) (thanks to \eqref{eq:LIP_in_H11} and to Remark \ref{rmk:BV_vs_H11}).
Notice also that the functional \(\tilde{\mathcal E}_{\mathscr A}\) can be expressed as
\begin{equation}\label{eq:tilde_E_as_compos}
\tilde{\mathcal E}_{\mathscr A}=\|\cdot\|_{L^1(T^*\X;\mathscr A)}\circ\d_{\mathscr A}.
\end{equation}
We denote by \({\sf n}_{\mathscr A}^*\) the Fenchel conjugate of \({\sf n}_{\mathscr A}\coloneqq\|\cdot\|_{L^1(T^*\X;\mathscr A)}\colon L^1(T^*\X;\mathscr A)\to[0,+\infty)\), i.e.
\[
{\sf n}_{\mathscr A}^*(V)\coloneqq\sup_{\omega\in L^1(T^*\X;\mathscr A)}\big(V(\omega)-\|\omega\|_{L^1(T^*\X;\mathscr A)}\big)\quad\text{ for every }V\in L^1(T^*\X;\mathscr A)'.
\]
Straightforward computations give that
\begin{equation}\label{eq:conj_of_n_A}
{\sf n}_{\mathscr A}^*(V)=\left\{\begin{array}{ll}
0\\
+\infty
\end{array}\quad\begin{array}{ll}
\text{ for every }V\in L^1(T^*\X;\mathscr A)'\text{ with }\|V\|_{L^1(T^*\X;\mathscr A)'}\leq 1,\\
\text{ for every }V\in L^1(T^*\X;\mathscr A)'\text{ with }\|V\|_{L^1(T^*\X;\mathscr A)'}>1.
\end{array}\right.
\end{equation}
\begin{theorem}\label{thm:BV_W_in_BV_H}
Let \((\X,\sfd,\mm)\) be a metric measure space. Let \(\mathscr A\) be a unital separating subalgebra of \(\LIP_\star(\X)\).
Then it holds that \({\rm BV}_{\rm W}(\X;\mathscr A)\subseteq{\rm BV}_{\rm H}(\X;\mathscr A)\) and
\[
\|{\bf D}f\|_{*,\mathscr A}\leq|{\bf D}f|_{\mathscr A}(\X)\quad\text{ for every }f\in{\rm BV}_{\rm W}(\X;\mathscr A).
\]
\end{theorem}
\begin{proof}
Fix any \(f\in{\rm BV}_{\rm W}(\X;\mathscr A)\). Since \(\mathcal E_{\mathscr A}\) is convex and \(L^1(\mm)\)-lower semicontinuous,
it coincides with the double Fenchel conjugate \(\tilde{\mathcal E}_{\mathscr A}^{**}\coloneqq(\tilde{\mathcal E}_{\mathscr A}^*)^*\)
of \(\tilde{\mathcal E}_{\mathscr A}\). Therefore, using \eqref{eq:tilde_E_as_compos}, \cite[Theorem 5.1]{BBS},
\eqref{eq:conj_of_n_A} and Lemma \ref{lem:D(d*)_to_Der}, we obtain that
\[\begin{split}
\mathcal E_{\mathscr A}(f)&=\tilde{\mathcal E}_{\mathscr A}^{**}(f)
=\sup_{g\in L^\infty(\mm)}\bigg(\int fg\,\d\mm-\tilde{\mathcal E}_{\mathscr A}^*(g)\bigg)
=\sup_{g\in L^\infty(\mm)}\bigg(\int fg\,\d\mm-({\sf n}_{\mathscr A}\circ\d_{\mathscr A})^*(g)\bigg)\\
&=\sup_{g\in L^\infty(\mm)}\bigg(\int fg\,\d\mm-\inf_{\substack{V\in D(\d_{\mathscr A}^*): \\ \d_{\mathscr A}^*V=g}}{\sf n}_{\mathscr A}^*(V)\bigg)
=\sup_{V\in D(\d_{\mathscr A}^*)}\bigg(\int f\,\d_{\mathscr A}^*V\,\d\mm-{\sf n}_{\mathscr A}^*(V)\bigg)\\
& = \sup_{\substack{v \in L^\infty(T\X;\mathscr A) :\\ \textsc{Int}_{\mathscr A}(v) \in D(\d_{\mathscr A}^*),\, \|v\|_{L^\infty(T\X;\mathscr A)} \le 1}} \bigg(\int f \, \d_{\mathscr A}^*(\textsc{Int}_{\mathscr A}(v))\,\d\mm \bigg) \\
& = \sup_{\substack{v \in L^\infty(T\X;\mathscr A) :\\ \textsc{Int}_{\mathscr A}(v) \in D(\d_{\mathscr A}^*),\, \|v\|_{L^\infty(T\X;\mathscr A)} \le 1}} \bigg(-\int f \, \div(b_v)\,\d\mm\bigg)  \\
&\leq\sup_{\substack{b\in{\rm Der}^\infty_\infty(\X;\mathscr A): \\ |b|_{\mathscr A}\leq 1}}\bigg(-\int f\,\div(b)\,\d\mm\bigg)
=\sup_{\substack{b\in{\rm Der}^\infty_\infty(\X;\mathscr A): \\ |b|_{\mathscr A}\leq 1}}{\bf L}_f(b)(\X)=|{\bf D}f|_{\mathscr A}(\X).
\end{split}\]
This gives that \(f\in{\rm BV}_{\rm H}(\X;\mathscr A)\) and \(\|{\bf D}f\|_{*,\mathscr A}\leq|{\bf D}f|_{\mathscr A}(\X)\),
thus proving the statement.
\end{proof}
\begin{theorem}\label{thm:equivalence_Lip}
Let \((\X,\sfd,\mm)\) be a metric measure space. Assume that \((\X,\sfd)\) is either complete or a Radon space. Then it holds that
\[
{\rm BV}_{\rm H}(\X)={\rm BV}_{\rm W}(\X)
\]
and we have that \(|{\bf D}f|=|{\bf D}f|_*\) for every \(f\in{\rm BV}_{\rm H}(\X)\).
\end{theorem}
\begin{proof}
It follows from Lemma \ref{lem:BV_H_loc_in_BV_W} and Theorem \ref{thm:BV_W_in_BV_H}.
\end{proof}
\begin{theorem}\label{thm:equivalence_A}
Let \((\X,\sfd,\mm)\) be a metric measure space. Assume that \((\X,\sfd)\) is complete. Let \(\mathscr A\) be a good subalgebra
of \(\LIP_b(\X)\). Then it holds that
\[
{\rm BV}_{\rm H}(\X;\mathscr A)={\rm BV}_{\rm H}(\X)={\rm BV}_{\rm W}(\X;\mathscr A)={\rm BV}_{\rm W}(\X;\LIP_b(\X)).
\]
Moreover, we have that \(\|{\bf D}f\|_{*,\mathscr A}=\|{\bf D}f\|_*=|{\bf D}f|_{\mathscr A}(\X)= |{\bf D}f|_{\LIP_b(\X)}(\X)\) and \(|{\bf D}f|_*=|{\bf D}f|_{\mathscr A}=|{\bf D}f|_{\LIP_b(\X)}\)
for every \(f\in{\rm BV}_{\rm H}(\X)\).
\end{theorem}
\begin{proof}
First of all, observe that for any unital separating subalgebra \(\mathscr A\subseteq\LIP_b(\X)\) it holds that
\begin{equation}\label{eq:equivalence_aux}\begin{split}
&{\rm BV}_{\rm W}(\X;\mathscr A)\subseteq{\rm BV}_{\rm H}(\X;\mathscr A)\subseteq{\rm BV}_{\rm H}(\X)\subseteq{\rm BV}_{\rm W}(\X;\LIP_b(\X)),\\
&|{\bf D}f|_{\LIP_b(\X)}(\X)\leq\|{\bf D}f\|_*\leq\|{\bf D}f\|_{*,\mathscr A}\leq|{\bf D}f|_\mathscr A(\X)\quad\text{ for every }f\in{\rm BV}_{\rm W}(\X;\mathscr A).
\end{split}\end{equation}
The above inclusions (from left to right) and the above inequalities (from right to left) follow from Theorem \ref{thm:BV_W_in_BV_H}, \eqref{eq:incl_BV_H}
and Remark \ref{rmk:BV_H_loc_in_BV_W_bis}. If \(\mathscr A\) is a good subalgebra, we know from Corollary \ref{cor:BV_W_indep_A} that
\({\rm BV}_{\rm W}(\X;\LIP_b(\X)) \subseteq {\rm BV}_{\rm W}(\X;\mathscr A)\) and \(|{\bf D}f|_{\mathscr A}\leq |{\bf D}f|_{\LIP_b(\X)}\)
for every \(f\in{\rm BV}_{\rm W}(\X;\LIP_b(\X))\). Therefore, the above inclusions and inequalities are all equalities. The statement follows.
\end{proof}

We conclude by formulating a partial converse of Theorem \ref{thm:equivalence_A}:
\begin{proposition}\label{prop:weakly_good_subalgebra}
Let \((\X,\sfd,\mm)\) be a metric measure space. Let \(\mathscr A\) be a unital separating subalgebra of \(\LIP_\star(\X)\). If
\({\rm BV}_{\rm H}(\X;\mathscr A)={\rm BV}_{\rm H}(\X)\) and \(\|{\bf D}f\|_{*,\mathscr A}=\|{\bf D}f\|_*\) for every \(f\in{\rm BV}_{\rm H}(\X;\mathscr A)\),
then $\mathscr A$ is a \textbf{weakly good subalgebra}: for every $x \in \X$
and for every $r>0$, there exist \((f_n)_n\subseteq\mathscr A\) and \(\rho\in L^\infty(\mm)\) with \(0\leq\rho\leq 1\) such that
\(\|f_n-\sfd_x\wedge r\|_{L^1(\mm)}\to 0\) and \(\lip_a(f_n)\mm \rightharpoonup\rho\mm\) weakly.
\end{proposition}
\begin{proof}
Fix any \(x\in\X\) and \(r>0\). Given that \(\sfd_x\wedge r\in\LIP_\star(\X) \subseteq{ \rm BV}_{\rm H}(\X)={\rm BV}_{\rm H}(\X;\mathscr A)\)
and \(\|{\bf D}(\sfd_x\wedge r)\|_{*,\mathscr A}=|{\bf D}(\sfd_x\wedge r)|_*(\X)\), we can find \((f_n)_n\subseteq\mathscr A\) such that
\(f_n\to\sfd_x\wedge r\) in \(L^1(\mm)\) and \(\int\lip_a(f_n)\,\d\mm\to|{\bf D}(\sfd_x\wedge r)|_*(\X)\). It follows from Remark \ref{rmk:Portmanteau}
that \(\lip_a(f_n)\mm\rightharpoonup|{\bf D}(\sfd_x\wedge r)|_*\) weakly. Since \(|{\bf D}(\sfd_x\wedge r)|_*\leq\mm\) as a direct consequence of
the definition, we proved the statement. 
\end{proof}
\appendix
\section{The total variation measure}\label{app:measure}
We report here the results in \cite[Lemmata 5.2 and 5.4]{Amb:DiMa:14}
briefly explaining why they do not make use of the completeness nor separability of the ambient metric space, and also
that they still work with our notion of asymptotic slope and for the approximating class being
$\LIP_\star(\X)$ and not $\LIP_{loc}(\X)$.
\begin{lemma}\label{lem:joint}
Let $(\X, \sfd, \mm)$ be a metric measure space and let $M,N \subseteq \X$ be open sets such that
$\eta\coloneqq\sfd(M\setminus N, N \setminus M)>0$. Set
\[ 
\phi(x):= \frac{3}{\eta} \min \left \{\big(\sfd(x, N \setminus M)- \eta/3\big)^+, \frac{\eta}{3} \right \}
\quad \text{ for every }x \in M \cup N,
\]
and $H\coloneqq \phi^{-1}((0,1))$. Then $H$ is an open subset of $ M \cap N$ with
$\sfd(H, (M \setminus N)\cup(N \setminus M))>0$. Whenever $u \in \LIP_\star(M)$ and $v \in \LIP_\star(N)$, setting
\[
w(x)\coloneqq(1-\phi(x))u(x) + \phi(x) v(x)\quad\text{ for every } x \in M \cup N,
\]
(where we extended trivially $u$ and $v$ to the whole set $M \cup N$), we have that $w \in \LIP_\star(M \cup N)$ and moreover it holds:
\begin{align*}
    \int_{M \cup N} \lip_a (w) \, \d \mm \le \int_M \lip_a (u) \, \d \mm + \int_N \lip_a (v) \, \d \mm + \frac{3}{\eta} \int_{\overline{H}}|u-v| \, \d \mm. \\
    w \equiv u \text{ on a neighbourhood of $M \setminus N$}, \quad w \equiv v \text{ on a neighbourhood of $N \setminus M$}.\\
    \int_{M \cup N} |w-\sigma| \, \d \mm \le \int_M | u-\sigma| \, \d \mm + \int_N |v-\sigma | \, \d \mm \quad \text{ for every } \sigma \in L^1(M \cup N).
\end{align*}
\end{lemma}
\begin{proof}[Comment to the proof] The proof is verbatim the one of \cite[Lemma 5.4]{Amb:DiMa:14}, which relies only on the convexity inequality and the locality of the slope (which work also for $\lip_a$), and on the fact that the distance function is $1$-Lipschitz.
\end{proof}
\begin{proposition} Let $(\X, \sfd, \mm)$ be a metric measure space. Let $\tau_\sfd$ be the topology induced by $\sfd$. For $f \in L^1(\mm)$ we define the function $|{\bf D}f|_*: \tau_\sfd \to [0,+\infty]$ as 
\[
|{\bf D}f|_*(U)\coloneqq\inf\bigg\{\liminf_{n\to\infty}\int_U\lip_a(f_n)\,\d\mm\;\bigg|\;\LIP_\star(U)\ni f_n\to f|_U\text{ in }L^1(\mm|_U)\bigg\} \quad\text{ for every } U \in \tau_\sfd,
\]
with the convention that $|{\bf D}f|_*(\varnothing)\coloneqq 0$. Then $|{\bf D}f|_*$ satisfies the following properties:
\begin{enumerate}
    \item If $A_1, A_2 \in \tau_\sfd$ and $A_1 \subseteq A_2$, then $|{\bf D}f|_*(A_1) \le |{\bf D}f|_*(A_2)$.
    \item If $A_1, A_2 \in \tau_\sfd$, then $|{\bf D}f|_*(A_1 \cup A_2) \le |{\bf D}f|_*(A_1) +|{\bf D}f|_*(A_2)$, with equality if $A_1$ and $A_2$ are disjoint.
    \item If $(A_n)_n \subseteq \tau_\sfd$ and $A_n \subseteq A_{n+1}$ for every $n \in \N$, then
    \begin{equation}
        \lim_{n \to \infty} |{\bf D}f|_*(A_n) = |{\bf D}f|_* \left ( \bigcup_{n=1}^{\infty} A_n \right ).
    \end{equation}
\end{enumerate}
In particular the formula 
\begin{equation}\label{eq:exte}
|{\bf D}f|_*(E)\coloneqq\inf\big\{|{\bf D}f|_*(U)\;\big|\;U\subseteq\X\text{ open, }E\subseteq U\big\}\quad\text{ for every } E\subseteq\X,
\end{equation}
provides a $\sigma$-subadditive extension of $|{\bf D}f|_*$ whose collection of additive sets, in the sense of Carath\'{e}odory, contains $\mathcal{B}(X)$. Therefore $|{\bf D}f|_*\colon \mathcal{B}(X) \to [0,+\infty]$ is a Borel measure.
\end{proposition}
\begin{proof}[Comment to the proof] The proof is almost verbatim the one of \cite[Lemma 5.2]{Amb:DiMa:14}; given Lemma \ref{lem:joint}, it relies only on the definition of $|{\bf D}f|_*$ via relaxation, the monotone convergence theorem, integral estimates, and the Carath\'{e}odory criterion which works for metric spaces. The only difference is that, in the notation of the proof of \cite[Lemma 5.2]{Amb:DiMa:14}, one has to replace the functions $u_m$ therein with $\tilde{u}_m:=u_m \wedge m$, and it is easy to check that $(\tilde{u}_m)_m \subset \LIP_\star(A)$.
\end{proof}
\medskip
\paragraph{\em\bfseries Data availability statement}
Data sharing not applicable to this article as no datasets were generated or analysed during the current study.
\medskip
\paragraph{\em\bfseries Conflict-of-interests statement}
The authors have no competing interests to declare that are relevant to the content of this article.
\end{document}